\documentclass[12pt, reqno]{amsart}
\usepackage[margin=1in]{geometry}

\usepackage{amssymb}
\usepackage{amsmath}
\usepackage{mathtools}
\usepackage{comment}
\usepackage{bm}

\usepackage{comment}
\usepackage{hyperref}
\usepackage{mathrsfs}
\usepackage{amsrefs}
\usepackage{bbm}
\usepackage{cite}
\usepackage{appendix}
\usepackage{amsfonts,mathtools}
\usepackage{bbm}

\usepackage{amsmath}
\usepackage{flexisym}
\usepackage{breqn}
\usepackage{mathfixs}
\usepackage{nicefrac}

\usepackage{amsthm,enumitem, physics, tikz}

\renewcommand{\Re}{\mathrm{Re}}
\renewcommand{\Im}{\mathrm{Im}}
\newcommand{\mc}{\mathcal}
\newcommand{\ms}{\mathscr}

\newcommand{\eps}{\varepsilon}

\numberwithin{equation}{section}
\newcommand{\Q}{\mathbb{Q}}
\newcommand{\R}{\mathbb{R}}
\newcommand{\N}{\mathbb{N}}
\newcommand{\C}{\mathbb{C}}

\usepackage{amsthm}

\newtheorem{theorem}{Theorem}[section]
\newtheorem{lemma}[theorem]{Lemma}
\newtheorem{corollary}[theorem]{Corollary}

\renewcommand{\bar}[1]{\overline{#1}}
\title{An Explicit Vinogradov--Korobov Zero-Free Region for Dirichlet $L$-functions}
\author{Tanmay Khale}
\begin{document}
\maketitle
\begin{abstract}
    We establish the first explicit form of the Vinogradov--Korobov zero-free region for Dirichlet $L$-functions.
\end{abstract}
\section{Introduction and statement of main results}
Vinogradov \cite{vinogradov} and Korobov \cite{korobov} proved that there exists an absolute and effectively computable constant $c>0$ such that the Riemann zeta function $\zeta(\sigma+it)$ does not vanish in the region
\begin{equation}\label{fordeq}
\sigma > 1 - \frac{c}{(\log |t|)^{\frac{2}{3}} (\log \log |t|)^{\frac{1}{3}}}, \qquad \abs{t} \geq 10.
\end{equation}
Ford \cite{fordzfr0}\footnote{The results of this paper depend on the calculations in \cite{fordzfr0}. An updated version with minor numerical corrections can be found in \cite{fordzfr}.} proved that one may take $c = 1/57.54$.  If $|t|$ is sufficiently large, then this improves to $1/49.13$.  In this paper, we prove the first explicit version of \eqref{fordeq} for Dirichlet $L$-functions.

\begin{theorem}\label{main_thm}
Let $q \geq 3$, and let $\chi\pmod{q}$ be a Dirichlet character. The Dirichlet $L$-function $L(\sigma+it,\chi)$ does not vanish in the region
\begin{equation}\label{smallt}
    \sigma \geq 1-\frac{1}{10.5 \log q+61.5(\log |t|)^{2 / 3}(\log \log |t|)^{1 / 3}}, \quad|t| \geq 10.
\end{equation}
Also, there exists an absolute and effectively computable constant $Y > 0$ such that $L(\sigma+it,\chi)$ does not vanish in the region
\begin{equation}\label{larget}
\sigma \geq 1- \frac{1}{ 10.1\log q + 49.13(\log |t|)^{\frac{2}{3}}(\log\log|t|)^{\frac{1}{3}}},\qquad \abs{t} \geq Y.
\end{equation}
\end{theorem}
As part of our proof of Theorem \ref{main_thm}, we first require that there exists {\it some} absolute and effectively computable constant $c'>0$ such that $L(\sigma+it,\chi)\neq 0$ in the region
\begin{equation}
\label{eqn:inexplicit_ZFR}
\sigma\geq 1-\frac{c'}{\log q + (\log|t|)^{\frac{2}{3}}(\log\log|t|)^{\frac{1}{3}}},\qquad \abs{t} \geq 10.
\end{equation}
Montgomery \cite[p. 176]{montgomery} states this result without reference or proof. A zero-free region of the form \eqref{eqn:inexplicit_ZFR} with weaker $q$-dependence follows from Mitsui \cite[Lemma 11]{Mitsui}. Several authors (e.g.,  \cite{Bartz,coleman,Hinz}) have proved an analogue of \eqref{eqn:inexplicit_ZFR} for $L$-functions of Hecke characters over a number field $K$, and the result for Dirichlet characters follows by taking $K=\mathbb{Q}$. For completeness, we give a short proof of \eqref{eqn:inexplicit_ZFR} in in Appendix \ref{appB} that relies only on well-known bounds for the sum in \eqref{eqn:Vinogradov_exponential_sum} along with a convenient version of Jensen's formula from complex analysis.
\subsection*{Acknowledgements}
The author thanks Jesse Thorner both for suggesting this project and for his extensive feedback which improved the quality of the presentation in this paper. The author also thanks Kevin Ford for numerous enlightening conversations. All numerical computations were carried out using Mathematica 12.
\section{Notation and Preliminaries}
\subsection{Preliminaries on Dirichlet $L$-functions}
The results in 
this subsection can be found in \cite{davenport}. 
Let $\chi\pmod{q}$ be a Dirichlet character. The Dirichlet $L$-function $L(s,\chi)$ is defined by the absolutely convergent Dirichlet series
\begin{equation*}
    L(s,\chi) = \sum_{n=1}^\infty \frac{\chi(n)}{n^s}
\end{equation*}
in the half-plane $\Re(s) > 1$. If $\chi$ is nonprincipal, then $L(s,\chi)$ extends to an entire function, while if $\chi$ is principal, then $L(s,\chi)$ extends to a meromorphic function on the complex plane with a single simple pole at $s=1$.

We say that $\chi$ is odd (resp. even) if $\chi(-1) = -1$ (resp. $\chi(-1)=1$). We write
\[
\mathfrak{a}{(\chi)}=\begin{cases}0 & \text { when } \chi \text { is even,} \\ 1 & \text { when } \chi \text { is odd.}\end{cases}
\]
When there is no ambiguity concerning the Dirichlet character under consideration, we write $\mathfrak{a}$ instead of $\mathfrak{a}(\chi)$. 
We also write
\[
\delta(\chi)= \begin{cases}0 & \text { when } \chi \text { is nonprincipal,} \\ 1 & \text { when } \chi \text { is principal.}\end{cases}
\]

Given $L(s,\chi)$ we define the \textit{completed} $L$-function $\xi(s,\chi)$ by
\begin{equation}\label{completed}
    \xi(s,\chi) = (s(s-1))^{\delta(\chi)} \Big( \frac{q}{\pi}\Big)^{\frac{s + \mathfrak{a}}{2}} \Gamma\Big(\frac{s + \mathfrak{a}}{2} \Big) L(s,\chi).
\end{equation}
If $\chi$ is primitive, there exists a constant $W(\chi)$ of modulus one such that the completed $L$-function satisfies the functional equation
\begin{equation}\label{fe}
    \xi(1-s, \bar{\chi})=W(\chi) \xi(s, \chi).
\end{equation}
Furthermore, $\xi(s,\chi)$ is an entire function of order one, and hence possesses a Hadamard factorization 
\begin{equation*}
    \xi(s,\chi) = \prod_{\xi(\rho,\chi) = 0} e^{\frac{s}{\rho}}\Big(1 - \frac{s}{\rho} \Big).
\end{equation*}
It follows from the analytic continuation and functional equation for $\xi(s,\chi)$ that if $\chi$ is primitive, then $L(s,\chi)$ has trivial zeros at $(-2n-2\delta(\chi)-\mathfrak{a}(\chi))_{n=0}^{\infty}$. If $\chi\pmod{q}$ is induced by the primitive character $\chi'\pmod{q'}$, then
\[
L(s,\chi) = L(s,\chi') \prod_{\substack{p \mid q \\ p \nmid q'}} \Big(1 - \frac{1}{p^s} \Big),
\]
and $L(s,\chi)$ has zeros on the line $\Re(s) = 0$ corresponding to the zeros of the factors $1 - p^{-s}$. We define
\[
\ms{I}(\chi)
=
\begin{cases}
0 & \text{if } \chi \text{ is primitive,} \\
1 & \text{otherwise.}
\end{cases}
\]

Observe that since $L(s,\bar{\chi}) = \bar{L(\bar{s},\chi)}$, the zeros of $L(s,\chi)L(s,\bar{\chi})$ come in conjugate pairs. Consequently, if the region $R \subset \C$ is free of zeros of $L(s,\chi)L(s,\bar{\chi})$, the reflection $\bar{R}$ is also zero-free. It therefore suffices to demonstrate a zero-free region holding for arbitrary $\chi$ and $\Im(s) \geq 0$ to conclude a zero-free region holding for arbitrary $\chi$ and $\Im(s)$ positive or negative. Throughout the remainder of the paper we will impose the restriction $\Im(s) \geq 0$ without loss of generality. By McCurley's \cite[Theorem 1.1]{mccurley} explicit version of a classical zero-free region of de la Vall\'ee Poussin, $L(\sigma+it,\chi)$ is non-vanishing for $t \geq 0$ and
\begin{equation}\label{mccurley}
    \sigma \geq 1- \frac{1}{ 9.64590880\log(\max\{q,qt,10\})}.
\end{equation}
apart from a possible real zero which is necessarily simple.
\subsection{Estimates for exponential sums}
Let $N \in \N$ and $u,t \in \R$, satisfying
\[
0<u\leq 1,\quad 1\leq N\leq t,\quad \lambda = \frac{\log t}{\log N}.
\]
Throughout this paper, we will assume that there exist constants $A,B > 0$ such that the Hurwitz zeta function, defined for $\sigma > 1$ by
\[
\zeta(\sigma+it,u)=\sum_{n=0}^\infty (n+u)^{-(\sigma+it)},
\]
and extended to a meromorphic function on the complex plane by analytic continuation, satisfies the bound
\begin{align}\label{eq1.2}
    |\zeta(\sigma+i t, u)-u^{-(\sigma+it)}| \leq A t^{B(1-\sigma)^{3 / 2}} (\log t)^{2 / 3}, \qquad t \geq 3,~1/2\leq \sigma \leq 1 .
\end{align}
Ford \cite[Theorem 1]{ford} shows that one may take $A = 76.2$ and $B = 4.45$ in \eqref{eq1.2}. More generally, in \cite[Lemma 7.3]{ford}, he establishes how the values of $A,B$ above explicitly depend on the constants in an upper bound for the exponential sums
\begin{equation}\label{eqn:Vinogradov_exponential_sum}
S(N, t)=\max _{0<u \leq 1} \max _{N<R \leq 2 N}\Big|\sum_{N<n \leq R}(n+u)^{-i t}\Big|.
\end{equation}
Namely, suppose that there exist constants $C,D>0$ such that $S(N, t) \leq C N^{1-1 /(D \lambda^{2})}$.
Let $B=\frac{2}{9} \sqrt{3 D}$. 
Assume further that $t \geq 10^{100}$ and $\frac{15}{16} \leq \sigma \leq 1$. By \cite[Lemma 7.3]{ford}, we then have that
\[
\begin{gathered}
|\zeta(\sigma+it, u)-u^{-(\sigma+it)}| \leq\Big(\frac{C+1+10^{-80}}{(\log  t)^{2/3}}+1.569 C D^{1 / 3}\Big) t^{B(1-\sigma)^{3 / 2}} (\log t)^{2/3}.
\end{gathered}
\]
\subsection{Notation}
We let
\[
\gamma_{\mathbb{\Q}} = \lim_{N\to\infty}\Big(\sum_{n=1}^N n^{-1}-\log N\Big) = 0.5772156649\dots
\]
denote the Euler--Mascheroni constant $\gamma_{\mathbb{\Q}}$. We write $\log_3(x) = \log\log\log(x)$ to denote the third-iterate of the logarithm function. Finally, any sum written in the form
\begin{equation*}
    \sum_{L(\rho,\chi) = 0} f(\rho)
\end{equation*}
in this paper is always a sum over zeros $\rho$ satisfying $0 < \Re(\rho)  < 1$, each counted with multiplicity.
\section{General Results}
Theorem \ref{main_thm} is a consequence of the following more general result.
\begin{theorem}
\label{theorem2}
Assume that there exist  constants $A>0$ and $0<B \leq 4.45$ such that \eqref{eq1.2} holds. Let $T_{0} \geq e^{1938}$ satisfy $\frac{\log T_{0}}{\log \log T_{0}} \geq \frac{1139}{B}$.  For $t\geq T_0-1$, define
\[
X(t)=1.16 \log (A+1)+0.845593+0.23531 \log \Big(\frac{B}{\log \log \tau}\Big)+\Big(\frac{1.694}{B^{4 / 3}}-\frac{2.087}{B^{1 / 3}}\Big)\Big(\frac{\log \log \tau}{\log \tau}\Big)^{\frac{1}{3}},
\]
where $\tau = 4t+1$, and
\[
M_{1}=\min \Big(0.055071, \frac{0.16521 -\frac{0.184833}{\log T_0}}{2.9997+\max _{t \geq T_{0}} X(t) / \log \log t}\Big).
\]
If for all characters $\chi\pmod{q}$ we have that  $L(\sigma+it,\chi)$ does not vanish in the region $t\in[T_0-1,T_0]$ and
\begin{equation}
\label{eq1.3}
\sigma \geq 1-\frac{1}{(10.082 + \frac{1.607}{\log\log T_0})\log q+ 9.791\log\log q + B^{2/3} M_1^{-1}(\log t)^{2 / 3}(\log \log t)^{1 / 3}},
\end{equation}
then for all $t \geq T_0$ and all $\chi\pmod{q}$, $L(\sigma+it,\chi)$ is nonvanishing in the region \eqref{eq1.3}.  
\end{theorem}
Theorem \ref{theorem2} has the following corollary for large $t$.
\begin{corollary}
\label{largetcorollary}
For all $0 < B \leq 4.45$, there exists an effectively computable constant $T_0(B)$ depending at most on $B$ such that $L(\sigma+it,\chi)$ is nonvanishing in the region
\begin{equation}\label{corollaryeq}
    \sigma \geq 1 -\frac{1}{10.1\log q + (0.05507)^{-1} B^{2/3} (\log |t|)^{2/3} (\log\log|t|)^{1/3}}, \qquad t \geq T_0(B).
\end{equation}
\end{corollary}
\begin{proof}

We first show that if $T_0 = T_0(B)$ is sufficiently large and if \eqref{eq1.3} holds for all $t \geq T_0$, then \eqref{corollaryeq} holds in this range. Observe that since $0.055071 \leq \frac{0.16520}{2.9997}$ and $X(t)$ is bounded, $M_1 = 0.055071$ for $t \geq T_0$ and $T_0$ sufficiently large. Next, in order to bound the $9.791 \log\log q$ term in the denominator of \eqref{eq1.3} from above, we split into the two cases $q \geq Q_0$ and $q < Q_0$, where $Q_0$ is a sufficiently large absolute constant. If $T_0$ and $Q_0$ are sufficiently large (where the thresholds for both $T_0$ and $Q_0$ are absolute), then for all $q \geq Q_0$ we have that $10.082 + \frac{1.607}{\log\log T_0} + 9.791\frac{\log \log q}{\log q} \leq 10.1$. On the other hand, if $q < Q_0$ and $T_0$ is sufficiently large (in terms of $Q_0$ and $B$), then
\[
\log\log q < \log\log Q_0 < (0.05507^{-1}-0.055071^{-1}) B^{2/3} (\log |t|)^{2/3} (\log \log|t|)^{1/3}.
\]
These observations together imply that if $T_0$ is sufficiently large and $t \geq T_0$, then the denominator in \eqref{eq1.3} is bounded above by $10.1 \log q + (0.05507)^{-1}B^{2/3}(\log |t|)^{2/3} (\log\log |t|)^{1/3}$. Consequently, Theorem \ref{theorem2} implies Corollary \ref{largetcorollary} as long as for some positive $T_0$, the condition that \eqref{eq1.3} holds for all zeros $\beta+i\gamma$ with $T_0-1 \leq \gamma \leq T_0$ is realized. We will show that this is indeed the case using a log-free zero-density estimate for Dirichlet $L$-functions.

Define $N(\sigma,T,\chi)=\#\{\rho = \beta+i \gamma\colon  L(\rho,\chi)=0,~\sigma \leq \beta \leq 1, |\gamma|\leq T\}$, where each zero $\beta + i\gamma$ is counted with multiplicity. Write $c = (0.055071)^{-1}$. Let $\eps > 0$ be arbitrary. By \cite[Theorem 1]{jutila}, there exists an effectively computable constant $c_{\eps}>0$ such that in the range $4/5 \leq \sigma \leq 1$ and $T \geq 1$, we have 
\begin{equation*}
    N(\sigma,T,\chi) \leq c_\eps (qT)^{(2+\eps)(1-\sigma)}.
\end{equation*}
Taking $\sigma = 1 - (10.2 \log q + (c/2) B^{2/3} (\log T)^{2/3} (\log \log T)^{1/3})^{-1}$, we find that if $T$ is sufficiently large, then
\begin{equation*}
    N(\sigma,T,\chi) \leq c_\eps e^{\frac{2+\eps}{10.2}} \exp\Big(\frac{(2+\eps)}{(c/2) B^{2/3}} \Big(\frac{\log T}{\log\log T} \Big)^{1/3}\Big) \leq \frac{T}{2}.
\end{equation*}
As such, if $T$ is sufficiently large, then there exists $T_0\in [\max\{e^{1938},\sqrt{T}\},T]$ 
such that all zeros $\beta + i\gamma$ of $L(s,\chi)$ with $T_0-1 \leq \gamma \leq T_0$ satisfy
\begin{multline*}
\sigma < 1- (10.2\log q + (c/2) B^{2/3} (\log T)^{2/3} (\log \log T)^{1/3})^{-1}\\
\leq 1- (10.2\log q + c B^{2/3} (\log \gamma)^{2/3} (\log \log \gamma)^{1/3})^{-1},
\end{multline*}
as desired.
\end{proof}
We now demonstrate how Theorem \ref{main_thm} may be deduced from Theorem \ref{theorem2}.
\begin{proof}[Proof of Theorem \ref{main_thm}, assuming Theorem \ref{theorem2}]
Note that \eqref{larget} is immediate from Corollary \ref{largetcorollary} with $B = 4.45$. It remains to demonstrate that 
\eqref{smallt} follows from Theorem \ref{theorem2}. Note that \eqref{larget} is immediate from Corollary \ref{largetcorollary} with $B = 4.45$. It remains to demonstrate that 
\eqref{smallt} follows from Theorem \ref{theorem2}. We will show the slightly stronger statement that $L(s,\chi)$ is nonvanishing in the region
\begin{equation}\label{slightlystronger}
    \sigma \geq 1-\frac{1}{10.3 \log q+9.791 \log \log q+61.306(\log t)^{2 / 3}(\log \log t)^{1 / 3}}, \quad|t| \geq 10
\end{equation}
To see that Theorem \ref{main_thm} follows from \eqref{slightlystronger}, suppose that $t \geq e^{1944}$, since Theorem \ref{main_thm} is a consequence of \eqref{mccurley} when $t \leq e^{1944}$. If $q < e^{428}$, then $9.791\log\log q < 9.791\log(428) < 0.194 (\log t)^{2/3} (\log \log t)^{1/3}$. On the other hand, if $q \geq e^{428}$, then $9.791\log\log q < 0.13861\log q$. In either case, \eqref{smallt} follows from \eqref{slightlystronger}. 

Observe that 
if $t \geq T_0 \geq e^{1944}$, $A = 76.2$, $B=4.45$ and $\tau = 4t+1$, then 
\begin{align*}
    &1.16\log (A+1)+0.845593+0.23531 \log \Big(\frac{B}{\log \log \tau}\Big)+\Big(\frac{1.694}{B^{4 / 3}}-\frac{2.087}{B^{1 / 3}}\Big)\Big(\frac{\log \log \tau}{\log \tau}\Big)^{\frac{1}{3}} \\
    &\leq 6.238712 -0.23531\log\log\log \tau -1.037389\Big(\frac{\log\log \tau}{\log \tau} \Big)^{1/3} \leq 5.61718,
    \end{align*}
since the first derivative test shows that the middle expression above is maximized at $\tau \approx e^{24\,012.6} $ on the domain $\tau \geq e^{1944}$. Theorem \ref{theorem2} implies that for $T_0 \geq e^{1944}$,
\begin{equation}\label{m1}
    M_1 \geq \min \Big(0.055071, \frac{0.16511492}{2.9997+\frac{5.61718}{\log \log T_0}}\Big).
\end{equation}
Moreover, if $T_0 \geq e^{1944}$, then
\begin{equation}\label{qconstant}
    10.082 + \frac{1.607}{\log\log T_0} \leq 10.3.
\end{equation}
For $t \leq e^{1944}$, \eqref{slightlystronger} follows from \eqref{mccurley}. On the other hand, for $t \geq e^{1944}$, we apply Theorem \ref{theorem2} with $A = 76.2$, $B = 4.45$, and $T_0 = e^{1944}$, utilizing \eqref{m1} and \eqref{qconstant} to bound the constants in front of $(\log t)^{2/3} (\log \log t)^{1/3}$ and $\log q$ respectively.
\end{proof} 
We note that if the unpublished result \cite[(17)]{kadiri} is used in place of \cite[Theorem 1.1]{mccurley}, the constant $61.306$ in \eqref{slightlystronger} may be lowered to $59.521$.
\section{Basic Estimates for $|\zeta(s)|$ and $L(s,\chi)$}
First, we will require the following elementary estimates for $L(s,\chi)$.
\begin{lemma}\label{zetabound}
If $1<\sigma \leq 1.06$ and $t$ is real, then for any Dirichlet character $\chi\pmod{q}$,  
\[
\frac{1}{\zeta(\sigma)} \leq|L(\sigma+i t,\chi)| \leq \zeta(\sigma) \leq 0.6+\frac{1}{\sigma-1}.
\]
For all $\sigma > 1$ and real $t$, we have
\[
\Big|-\frac{L^{\prime}}{L}(\sigma+i t,\chi)\Big|<\frac{1}{\sigma-1} .
\]
\end{lemma}
\begin{proof}
This follows immediately from \cite[Lemma 3.1]{fordzfr}.
\end{proof}
\begin{lemma}\label{lemma3.2}
Let $\chi\pmod{q}$ be a primitive Dirichlet character, with $q \geq 3$, and let $u$ be real. Then, we have
\begin{align*}
    \Big|\frac{L'}{L}(-\textstyle{\frac{1}{2}}+iu,\chi)\Big| \leq  8.21 + \log q + \frac{1}{2} \log (1+ u^2/4).
\end{align*}
\end{lemma}
\begin{proof}
The functional equation for Dirichlet $L$-functions (see \eqref{completed}, \eqref{fe}) implies that 
\begin{align*}
    \frac{1}{2}\log\Big(\frac{q}{\pi}\Big) + \frac{1}{2}\frac{\Gamma'}{\Gamma}\Big(\frac{s + \mathfrak{a}}{2}\Big) + \frac{L'}{L}(s,\chi) = -\frac{1}{2}\log\Big(\frac{q}{\pi}\Big) - \frac{1}{2}\frac{\Gamma'}{\Gamma}\Big(\frac{1-s + \mathfrak{a}}{2}\Big) - \frac{L'}{L}(1-s,\bar{\chi}),
\end{align*}
or, by rearranging,
\begin{align*}
    -\frac{L'}{L}(s,\chi) = \frac{L'}{L}(1-s,\bar{\chi})  +  \log\Big(\frac{q}{\pi}\Big)  +  \frac{1}{2}\Big(\frac{\Gamma'}{\Gamma}\Big(\frac{s + \mathfrak{a}}{2} \Big) + \frac{\Gamma'}{\Gamma}\Big(\frac{1-s + \mathfrak{a}}{2} \Big) \Big).
\end{align*}
If $s=-\frac{1}{2}+iu$, then it follows from the functional equation $\Gamma(z+1)=z\Gamma(z)$ and the identity $\overline{\Gamma(z)}=\Gamma(\bar{z})$ twice that
{\small\begin{align}
\frac{1}{2}\Big(\frac{\Gamma'}{\Gamma}\Big(\frac{s+\mathfrak{a}}{2}\Big)+\frac{\Gamma'}{\Gamma}\Big(\frac{1-s+\mathfrak{a}}{2}\Big)\Big)&=\mathrm{Re}\Big(\frac{\Gamma'}{\Gamma}\Big(\frac{3}{4}+\frac{\mathfrak{a}}{2}+\frac{iu}{2}\Big)\Big) +\frac{1}{\frac{1}{2}-\mathfrak{a}-iu} \label{firstfe}\\
&= \mathrm{Re}\Big(\frac{\Gamma'}{\Gamma}\Big(\frac{7}{4}+\frac{\mathfrak{a}}{2}+\frac{iu}{2}\Big)  - \frac{1}{\frac{3}{4} + \frac{\mathfrak{a}}{2} + \frac{iu}{2}}\Big) +\frac{1}{\frac{1}{2}-\mathfrak{a}-iu}. \label{secondfe}
\end{align}}%

Ono and Soundararajan show in \cite[Lemma 4]{sound} that if $\mathrm{Re}(z)\geq 1$, then
\[
\Big|\frac{\Gamma^{\prime}}{\Gamma}(z)\Big| \leq \frac{11}{3}+\frac{\log (1+x^2 )}{2}+\frac{\log (1+y^2 )}{2} .
\]
We use the inequality $|\log(q/\pi)| \leq \log(q)-1$ (which holds for $q \geq 3$), combined with \eqref{firstfe} if $\mathfrak{a}=1$ and \eqref{secondfe} if $\mathfrak{a}=0$ to find that
\begin{align*}
\Big|\frac{L'}{L}(-\tfrac{1}{2}+iu,\chi)\Big|&\leq \log(q) - 1 + \Big|\frac{L'}{L}\Big(\frac{3}{2}+iu,\bar{\chi}\Big)\Big| +\sqrt{\frac{4}{1+4u^2}} + \sqrt{\frac{16}{9 + 4u^2}} \\
&\qquad+ \frac{11}{3} + \frac{\log(1 + 49/16)}{2} + \frac{\log(1+u^2/4)}{2}.
\end{align*}
Since $|\frac{L'}{L}(\frac{3}{2}+iu,\bar{\chi})|\leq -\frac{\zeta'}{\zeta}(\frac{3}{2})\leq 1.505236$, the desired result follows.
\end{proof}
In standard treatments, the zeros of Dirichlet $L$-functions are counted ``symmetrically," i.e. in the region $\abs{t} \leq T$. However, we require an estimate for zeros of $L(s,\chi)$ in the \textit{upper half plane}, for which we could find no explicit estimates in the literature. Let 
\[
    N^{\pm}(T,\chi) = \#\{\rho = \beta + i\gamma: 0 < \pm \gamma < T\} + \frac{\#\{\rho = \beta+i\gamma: \gamma \in \{0,\pm T\}\}}{2}.
\]
\begin{lemma}\label{zerocount}
Let $\chi$ be a primitive Dirichlet character modulo $q \geq 3$. For all $T > 1$ and $0 < \eta \leq 1/2$,
\begin{equation*}
\Big|N^\pm(T,\chi) - \frac{T}{2\pi}\log \frac{qT}{2\pi e} - C_0(\chi,\eta)\Big| \leq C_1 \log(qT) + C_2,
\end{equation*}
where
\[
    C_1 = \frac{1 + 2\eta}{2\pi \log(2)}, \qquad C_2 = 0.1529-0.134 \eta+\frac{2 \log \zeta(1+\eta)}{\log 2}-\frac{ \log \zeta(2+2 \eta)}{\log 2} +\frac{2}{\pi} \log \zeta\Big(\frac{3}{2}+2 \eta\Big).
\]
\end{lemma}
\begin{proof}
By interchanging $\chi$ and $\bar{\chi}$, it suffices to prove the desired estimate for $N^+(T,\chi)$. If we define
\[
S(T,\chi):= \frac{1}{2\pi}\Big[\lim_{t\to T^+}\int_{\frac{1}{2}}^{\infty}\mathrm{Im}\Big(-\frac{L'}{L}(\sigma+it,\chi)\Big)d\sigma + \lim_{t\to T^-} \int_{\frac{1}{2}}^{\infty}\mathrm{Im}\Big(-\frac{L'}{L}(\sigma+it,\chi)\Big)d\sigma\Big],
\]
then by \cite[Theorem 14.5]{mv}, we have that
\begin{equation}\label{mvzerocount}
    N(T, \chi)=\frac{1}{\pi} \arg \Big(\Gamma\Big(\frac{1}{4}+\frac{\mathfrak{a}}{2}+i \frac{T}{2}\Big)\Big)+\frac{T}{2 \pi} \log \frac{q}{\pi}+S(T, \chi)-S(0, \chi).
\end{equation}
Let $\sigma_1 = \frac{3}{2} + 2\eta > 1$, and let $C_0(\chi,\eta) =  \frac{1}{\pi}\arg(L(\sigma_1,\chi))-S(0,\chi)$. 

By \cite[(2.8)]{mccurley} and the subsequent unlabeled display, we have that for $T \geq 1$,
\begin{equation}\label{gammabound}
\begin{aligned}
\Big|\arg \Big(\Gamma\Big(\frac{1}{4}+\frac{\mathfrak{a}}{2}+i \frac{T}{2}\Big)\Big)- \frac{T}{2} \log \frac{T}{2 e} \Big|\leq 0.6909.
\end{aligned}
\end{equation}
Next, we bound $|S(T,\chi)-\frac{1}{\pi}\arg(L(\sigma_1,\chi))|$ by considering the variation of $\arg(L(s,\chi))$ along the contour $\mc{C}: \sigma_1 \to \sigma_1+iT \to 1/2 + iT$. 
 By \cite[(2.15)]{mccurley}, we have for all $T \geq 1$ that
\begin{equation}\label{horizontal}
    \Big|\Delta_{\sigma_1 + iT \to \frac{1}{2} + iT}~\arg (L(s, \chi))\Big| \leq  \frac{1+2 \eta}{2 \log 2} \log (0.74685 q T) +\frac{2 \pi \log \zeta(1+\eta)}{\log 2}-\frac{\pi \log \zeta(2+2 \eta)}{\log 2} .
\end{equation}
Moreover, since $|\arg L(s, \chi)|\leqslant| \log L(s, \chi) \mid \leqslant \log \zeta(\Re(s))$ and $\sigma_1 = \frac{3}{2} + 2\eta$, we have
\begin{equation}\label{vertical}
    |\Delta_{\sigma_1 + it_0 \to \sigma_1 + iT}~{\arg(L(s,\chi))}| \leq 2 \log \Big(\zeta\Big(\frac{3}{2}+ 2\eta \Big) \Big).
\end{equation}
Combining \eqref{gammabound}, \eqref{horizontal}, and \eqref{vertical} yields the desired result.
\end{proof}
We obtain the following corollary.
\begin{corollary}\label{corollary}
Let $\chi$ be a primitive Dirichlet character modulo $q \geq 3$. For all $T > 1$,
\begin{equation*}
\abs{N^\pm(T,\chi) - \frac{T}{2\pi}\log \frac{qT}{2\pi e} - C_0(\chi)} \leq 0.2297 \log(qT) + 24.77,
\end{equation*}
where $C_0$ is as given in Lemma \ref{zerocount}.
\end{corollary}
\begin{proof}
Take $\eta = 0.00019$, in Lemma \ref{zerocount}.
\end{proof}
\section{ Additional Bounds for $|\zeta(s)|$ and $|L(s,\chi)|$}
Unlike the results of the previous section, the estimates in this section will depend on the explicit exponential sum estimates derived in Ford \cite{ford}. We first extract an upper bound for $L(s,\chi)$ from \eqref{eq1.2}.
\begin{lemma}\label{appliedford}
Assume that there exist positive constants $A$ and $B$ such that \eqref{eq1.2} holds. Let $q \geq 3$, and let $\chi\pmod{q}$ be a Dirichlet character. If $|t| \geq 3$ and $1/2 \leq \sigma \leq 1$, then
\begin{align*}
    |L(\sigma+it,\chi)| \leq  A q^{1-\sigma}|t|^{B(1-\sigma)^{3/2}}(\log |t|)^{2/3} + 1.92 \frac{q^{1-\sigma}-1}{1-\sigma}.
\end{align*}
\end{lemma}
\begin{proof}
Note by the periodicity of the values of $\chi(n)$ that
\begin{align*}
    L(\sigma+it,\chi) &= q^{-(\sigma+it)}\sum_{a=1}^q \chi(a) \zeta\Big(\sigma+it, \frac{a}{q}\Big).
\end{align*}
Hence, it follows from \eqref{eq1.2} that
\begin{align*}
    |L(\sigma+it,\chi)| &= \Big|q^{-(\sigma+it)}\sum_{a=1}^q \chi(a) \Big(\zeta\Big(\sigma+it, \frac{a}{q}\Big) - \Big(\frac{a}{q} \Big)^{-(\sigma+it)} \Big) + q^{-(\sigma+it)} \sum_{a=1}^q \Big(\frac{a}{q} \Big)^{-(\sigma+it)} \Big| \\
    &\leq A q^{1-\sigma} |t|^{B(1-\sigma)^{3 / 2}} (\log |t|)^{2/3} + \sum_{a=1}^q a^{-\sigma}.
\end{align*}
To bound the sum above, observe that
\begin{align*}
    \sum_{a=1}^q a^{-\sigma} \leq 1 + \int_1^{q} u^{-\sigma} du = 1 + \frac{u^{1-\sigma}}{1-\sigma}\Big\rvert_{u=1}^q = 1 + \frac{q^{1-\sigma}-1}{1-\sigma}.
\end{align*}
Furthermore, note that $\sigma \leq 1$ implies that $\frac{q^{1-\sigma}-1}{1-\sigma} = \int_1^q u^{-\sigma} du \geq \int_1^q \frac{du}{u} = \log q$. Since $q \geq 3$, it follows that $1 + \frac{q^{1-\sigma}-1}{1-\sigma} \leq 1.92 \frac{q^{1-\sigma}-1}{1-\sigma} $.
\end{proof}
This following lemma is a variant of \cite[Lemma 3.4]{fordzfr}.
\begin{lemma}\label{lemma3.4}
Let $q \geq 3$, and let $\chi\pmod{q}$ be a Dirichlet character. 
Fix $\sigma \in [1/2,1)$, and $a \in (0,1/2]$. Assume that there exist positive constants $A$ and $B$ such that \eqref{eq1.2} holds. Suppose that $t, q$ satisfy
\begin{align}\label{bound_condition}
     1-\sigma \geq 1.92 (\log(t/100))^{-2/3},\qquad t \geq \max\{q^{\frac{1}{100\,000}},e^{1938}\}.
\end{align}
Then, it follows that
\begin{multline}
\label{lemmaintegral}
\int_{-\infty}^{\infty} \frac{\log |L(\sigma+i t+i a u,\chi)|}{(\cosh u)^2} d u \leq 2\Big(\log(A+1) + (1-\sigma) \log q  \\
    +B(1-\sigma)^{3/2} \log t+\frac{2}{3} \log \log t\Big).
\end{multline}
\end{lemma}
\begin{proof}
Note that even if $L(\sigma+it+iau,\chi) = 0$ on the path of integration, the integral in \eqref{lemmaintegral} converges, since all zeros have finite order. Let 
\[S = \Big\{u \in \R:  \frac{1.92}{1-\sigma} \geq  |t+au|^{B(1-\sigma)^{3/2}} (\log |t+au|)^{2/3} \Big\}.\] 
Note that if $u \notin S$ and $|t+au| > 3$, then
\begin{align}\label{secondcasebound}
    |L(\sigma + it + iau,\chi)| \leq  (A + 1) q^{1-\sigma}|t+au|^{B(1-\sigma)^{3/2}}(\log |t+au|)^{2/3}.  
\end{align}
To see this, observe that if $u \notin S$, then
\begin{align*}
    1.92 \frac{q^{1-\sigma}-1}{1-\sigma} \leq q^{1-\sigma} |t+au|^{B(1-\sigma)^{3/2}} (\log|t+au|)^{2/3}.
\end{align*}
Consequently, \eqref{secondcasebound} follows from Lemma \ref{appliedford} (which is applicable since $|t+au| \geq 3$).  We now split up the the integral in \eqref{lemmaintegral} into six disjoint pieces. 

{\bf Case 1: $|t+au|\leq 3$.}  Note that in this range, $|\sigma + it + iau| \leq \sqrt{10}$. For $\chi$ nonprincipal, we use
\begin{align*}
    |L(s,\chi)| = |s| \Big|\int_1^\infty \sum_{n \leq x} \chi(n)  x^{-s-1} dx\Big| \leq |s| q \int_1^\infty x^{-3/2} dx =  2q|s|
\end{align*}
with $s = \sigma + it + iau$, which implies that $\log|L(\sigma+it +iau,\chi)| \leq \log (2q|\sigma+it+iau|) \leq \log(2\sqrt{10}q )$. By the assumptions $t \geq q^{\frac{1}{100\,000}}$ and $t \geq e^{1938}$, we have that 
\begin{equation}\label{case1}
|L(\sigma+it+iau,\chi)| \leq 100\,000\log t + \log(2\sqrt{10}) \leq 100\,001 \log t.
\end{equation}
On the other hand, if $\chi$ is principal, then the same conclusion follows from the identity (see  \cite[(2.1.4)]{tit})
\begin{equation*}
    \zeta(s)=\frac{1}{s-1}+\frac{1}{2}+s \int_{1}^{\infty} \frac{\lfloor x\rfloor-x+1 / 2}{x^{s+1}} d x,
\end{equation*}
combined with \eqref{bound_condition} and the fact that
{\small\[
    |\log|L(s,\chi_0)| - \log |\zeta(s)|| \leq  \sum_{p \mid q} \log\frac{1}{|1-p^{-\sigma}|} \leq  -\log \big((1-2^{-\frac{1}{2}})(1-3^{-\frac{1}{2}})\big) -\log \big(1-5^{-\frac{1}{2}}\big)\frac{\log q}{\log 5}.
\]}%

{\bf Case 2: $u \in S$ and $|t+au| > 3$.} Observe that
\[
\frac{q^{1-\sigma}-1}{1-\sigma}\leq\begin{cases}
(e-1)\log q&\mbox{if $\sigma\geq 1-(\log q)^{-1}$,}\\
q^{1-\sigma}\log q&\mbox{if $\sigma< 1-(\log q)^{-1}$.}
\end{cases}
\]
Thus, in either case, we have that $\frac{q^{1-\sigma}-1}{1-\sigma} \leq \sqrt{q} \log q$. Consequently, by the assumptions that $q \leq t^{100\,000}$ and $t \geq e^{1938}$, we have that
\begin{equation}\label{weirdsumbound}
\begin{aligned}
   \log\Big(1.92\frac{q^{1-\sigma}-1}{1-\sigma}\Big) \leq \log(1.92) + \frac{1}{2}\log q + \log\log q \leq 50\,001 \log t.
 \end{aligned}
\end{equation}
Moreover, by \eqref{bound_condition}, the assumption $t \geq e^{1938}$, and the fact that $u \in S$, we have that
\begin{equation}\label{mainthingbound}
\begin{aligned}
    \log(A q^{1-\sigma} |t+au|^{B(1-\sigma)^{3 / 2}} (\log |t+au|)^{2/3}) &\leq \log(A) + \frac{1}{2}\log q + \log((\log t)^{2/3}) \\
    &\leq  \log(A) + 50\,000\log(t) + \frac{2}{3}\log\log t\\
    &\leq \log(A) + 50\,001\log(t).
\end{aligned}
\end{equation}
It follows from Corollary \ref{appliedford}, \eqref{weirdsumbound} and \eqref{mainthingbound} that for $u \in S$ such that $|t+au| > 3$, we have
\begin{equation}\label{case2}
\begin{aligned}
    \log |L(\sigma + it + iau, \chi )| &\leq \log\Big( A q^{1-\sigma}|t+au|^{B(1-\sigma)^{3/2}}(\log |t+au|)^{2/3} + 1.92 \frac{q^{1-\sigma}-1}{1-\sigma}\Big) \\
    &\leq \log(2) + \log(A+1) + 50\,001\log(t)\\
    &\leq \log(A+1) + 50\,002 \log(t).
\end{aligned}
\end{equation}

{\bf Case 3: $\frac{-2t}{a} \leq u < \frac{-t-3}{a}$ and $u \notin S$}. In this case, \eqref{secondcasebound} gives the inequality
\begin{equation}\label{case4}
\begin{aligned}
|L(\sigma+it+iu,\chi)| \leq (A + 1) q^{1-\sigma} t^{B(1-\sigma)^{3/2}}(\log t)^{2/3}.
\end{aligned}
\end{equation}

{\bf Case 4: $u > \frac{3-t}{a}$ and $u \notin S$}. We use the inequalities $\log(1+x) \leq x$ and $\log(1+x) \leq x - \frac{1}{2}x^2 + \frac{1}{3}x^3$, both valid for $x > -  1$, alongside \eqref{secondcasebound} to obtain that
\begin{equation}\label{case5}
\begin{aligned}
\log |L(\sigma+i t+i a u,\chi)| &\leq \log (A + 1)  + (1-\sigma) \log q \\
&\qquad+B(1-\sigma)^{3/2} \log (t+a u)+\frac{2}{3} \log \log (t+a u) \\
&\leq \log \Big((A + 1) q^{1-\sigma}t^{B(1-\sigma)^{3/2}}(\log t)^{2/3}\Big)\\
&\qquad+\Big(B(1-\sigma)^{3/2}+\frac{2}{3\log t}\Big)\Big(\frac{a u}{t}-\frac{(a u)^{2}}{2 t^{2}}+\frac{(a u)^{3}}{3 t^{3}}\Big).
\end{aligned}
\end{equation}

{\bf Case 5: $u < -\frac{2t}{a}$ and $u\notin S$.} Using $\log(1+x) \leq x$ and \eqref{secondcasebound}, we find that
\begin{equation}\label{case6}
\begin{aligned}
    \log |L(\sigma+i t+i a u,\chi)| &\leq \log \Big((A+1)q^{1-\sigma} t^{B(1-\sigma)^{3/2}}(\log t)^{2/3}\Big)\\
    &\qquad+\Big(B(1-\sigma)^{3/2}+\frac{2}{3\log t}\Big)\Big(\frac{-a u-2 t}{t}\Big).
\end{aligned}
\end{equation}

The estimates \eqref{case1} and \eqref{case2}-\eqref{case6}, along with the identity $\int_{-\infty}^\infty (\sech (x))^2  dx = 2$, yield that
\begin{align*}
    &\int_{-\infty}^{\infty} \frac{\log |L(\sigma+i t+i a u,\chi)|}{(\cosh u)^2} d u \\
    &\qquad\leq 2\Big(\log(A+1) + (1-\sigma) \log q +B(1-\sigma)^{3/2} \log t +\frac{2}{3} \log \log t\Big)+E,
\end{align*}
where
\[
\begin{aligned}
E &= 50\,002\log(t)\int_{\substack{u \in S \\ |t+au| > 3}} \frac{1}{(\cosh u)^2} du + \frac{600\,006 \log (t)}{a (\cosh(\frac{3-t}{a}))^2} \\
&+\Big(B(1-\sigma)^{3/2}+\frac{2}{3\log t}\Big)\Big(\int_{-\infty}^{-\frac{2 t}{a}} \frac{-a u-2 t}{t (\cosh u)^2} d u+\int_{\frac{3-t}{a}}^{\infty} \frac{\frac{a u}{t}-\frac{(a u)^{2}}{2 t^{2}}+\frac{(a u)^{3}}{3 t^{3}}}{(\cosh u)^2} d u\Big) .
\end{aligned}
\]
By \eqref{bound_condition}, $S$ is a subset of $\{u: u \leq -0.99 t/a\}$. 
Since (for $M \geq 0$) we have that $\int_{u \leq -M} \frac{du}{(\cosh u)^2} \leq 2 e^{-2M}$, and since $a \leq 1/2$, it follows that
\begin{align*}
     50\,002\log(t)\int_{\substack{u \in S \\ |t+au| > 3}} \frac{1}{(\cosh u)^2} d u &\leq  100\,004\log(t) e^{-1.98t/a} \\
     &\leq e^{-t/a} \cdot e^{-1.96t} \cdot 100\,004 \log(t) \leq 0.001 e^{-t/a}  .
\end{align*}
Since $a \leq 1/2$ and $t \geq e^{1938}$, we have that $a e^{(t-6) / a} \geq 2 e^{2 t-12}$. Moreover, since $ (\cosh u)^2 \geq \frac{1}{4}e^{2|u|} $, we have that $a(\cosh(\frac{3-t}{a}))^2 \geq \frac{1}{4} a e^{(t-6)/a} e^{t/a} \geq \frac{1}{2} e^{2t-12} e^{t/a}$. Consequently, we find that
\begin{align*}
    0.001 e^{-t/a} + \frac{600\,006 \log (t)}{a (\cosh(\frac{3-t}{a}))^2} &\leq 0.001 e^{-t/a} +  \frac{1\,200\,012 \log (t)}{e^{t / a+2 t-12}} \leq e^{-t/a}.
\end{align*}
The final lines of the proof of \cite[Lemma 3.4]{fordzfr} show, \textit{mutatis mutandis}, that 
\small
\begin{align*}
    \Big(B(1-\sigma)^{3/2}+\frac{2}{3\log t}\Big)\Big(\int_{-\infty}^{-\frac{2 t}{a}} \frac{-a u-2 t}{t (\cosh u)^2} d u+\int_{\frac{3-t}{a}}^{\infty} \frac{\frac{a u}{t}-\frac{(a u)^{2}}{2 t^{2}}+\frac{(a u)^{3}}{3 t^{3}}}{(\cosh u)^2} d u\Big) \leq -\frac{a}{20t^2\log t}.
\end{align*}
\normalsize
It follows that $E \leq 0$.
\end{proof}
\section{Detecting zeros of $L(s,\chi)$}
We first recall the form of Jensen's formula shown for functions in a vertical strip proven in \cite[Lemma 2.2]{fordzfr}.

\begin{lemma}[Lemma 2.2 of \cite{ford}]\label{lemma2.2}
Suppose $f$ is the quotient of two entire functions of finite order, and does not have a zero or a pole at $z=z_0$ nor at $z=0$. Then, for all $\eta>0$ except for a set of Lebesgue measure 0 (the exceptional set may depend on $f$ and $z_0$), we have
\[
\begin{aligned}
-\Re \Big(\frac{f^{\prime}(z_0)}{f(z_0)} \Big)&=\frac{\pi}{2 \eta} \sum_{|\Re(z_0-\omega)| \leq \eta} m_\omega \Re\Big( \cot \Big(\frac{\pi(\omega-z_0)}{2 \eta}\Big)\Big) \\
&+\frac{1}{4 \eta} \int_{-\infty}^{\infty} \frac{\log |f(z_0-\eta+\frac{2 \eta i u}{\pi})|-\log |f(z_0+\eta+\frac{2 \eta i u}{\pi})|}{(\cosh  u)^2} d u.
\end{aligned}
\]
where $\omega$ runs over the zeros and poles of $f$, 
and where the multiplicity $m_{\omega}$ is positive if $\omega$ is a zero or negative if $\omega$ is a pole.
\end{lemma}

The following lemma is our main method for detecting zeros of $L(s,\chi)$.
\begin{lemma}\label{lemma4.1}
Assume that there exist positive constants $A$ and $B$ such that \eqref{eq1.2} holds. Let $q \geq 3$, and let $\chi\pmod{q}$ be a Dirichlet character. Let $t \geq \max\{e^{1938}, q^{\frac{1}{100\,000}}\}$, $0 < \eta  < 10$, $\sigma-\eta > 1/2$, $1 \leq \sigma \leq 1 + \eta - 1.92(\log(t/100))^{-2/3}$, and $s = \sigma + it$. If $S\subseteq\{z\in\mathbb{C}\colon\sigma -\eta \leq \Re(z) \leq 1\}$, 
then
\small
\[
\begin{aligned}
-\Re\Big(\frac{L^{\prime}(s,\chi)}{L(s,\chi)}\Big) &\leq - \sum_{\rho \in S, L(\rho,\chi)=0} \Re \Big(\frac{\pi}{2 \eta} \cot \Big(\frac{\pi(s-\rho)}{2 \eta}\Big)\Big) \\
&\quad+\frac{1}{2 \eta}\Big((1-\sigma+\eta)\log q +  \frac{2}{3} \log \log t+B(1-\sigma+\eta)^{3 / 2} \log t+\log(A+1)\Big) \\
&\quad-\frac{1}{4 \eta} \int_{-\infty}^{\infty} \frac{\log |L(s+\eta+2 \eta i u / \pi,\chi)|}{(\cosh u)^2} d u + (1-\mathfrak{a})(1-\delta(\chi))e^{-1937}.
\end{aligned}
\]
\normalsize
\end{lemma}
\begin{proof}
Apply Lemma \ref{lemma2.2} with $f = s^{(1-\delta(\chi))(\mathfrak{a}-1)}L(s,\chi)$ and $z_0 = s$. The conclusion follows precisely as in the proof of \cite[Lemma 4.1]{fordzfr} when $\mathfrak{a} = 1$ or $\delta(\chi) = 1$, with Lemma \ref{lemma3.4} in place of \cite[Lemma 3.4]{fordzfr}. When $\delta(\chi) = 0$ and $\mathfrak{a} = 0$, we have that
\begin{align*}
    &\frac{1}{4 \eta} \int_{-\infty}^{\infty} \frac{\log |(s-\eta + \frac{2\eta iu}{\pi})^{-1} L(s-\eta+\frac{2 \eta i u}{\pi},\chi)|-\log |(s + \eta + \frac{2\eta i u}{\pi})^{-1} L(s+\eta+\frac{2 \eta i u}{\pi},\chi )|}{(\cosh u)^2} d u \\
    &= \frac{1}{4 \eta} \int_{-\infty}^{\infty} \frac{\log |L(s-\eta+\frac{2 \eta i u}{\pi},\chi)|-\log |L(s+\eta+\frac{2 \eta i u}{\pi},\chi )|}{(\cosh u)^2} d u \\
    &\qquad+ \frac{1}{4\eta} \int_{-\infty}^\infty \log \Big|\frac{s+\eta+\frac{2\eta iu}{\pi}}{s - \eta + \frac{2\eta i u}{\pi}}\Big| \frac{du}{(\cosh u)^2}.
\end{align*}
Moreover, we have that
\begin{align*}
    \frac{1}{4\eta} \int_{-\infty}^\infty \log \Big|\frac{s+\eta+\frac{2\eta iu}{\pi}}{s - \eta + \frac{2\eta i u}{\pi}}\Big| \frac{du}{(\cosh u)^2} &= \frac{1}{4\eta} \int_{-\infty}^\infty \log \Big|1 + \frac{2 \eta}{s - \eta + \frac{2\eta iu}{\pi}} \Big| \frac{du}{(\cosh u)^2} \\
    &\leq \frac{1}{2} \int_{-\infty}^\infty \frac{1}{\abs{s - \eta + \frac{2\eta iu}{\pi}}} \frac{du}{(\cosh u)^2}.
\end{align*}
Finally, we crudely estimate the latter integral. Utilizing the facts that $t \geq e^{1938}$, $\eta \leq 10$, $\sigma-\eta > 1/2$, and that for $M \geq 0$, $\int_{u \leq -M}\frac{du}{(\cosh u)^2} \leq 2e^{-2M}$, we obtain
\begin{align*}
    &\int_{-\infty}^\infty \frac{1}{\abs{s - \eta + \frac{2\eta iu}{\pi}}} \frac{du}{(\cosh u)^2} \\
    &=\int_{\abs{t + \frac{2\eta u}{\pi}}< 0.5t} \frac{1}{\abs{s - \eta + \frac{2\eta iu}{\pi}}} \frac{du}{(\cosh u)^2} + \int_{\abs{t + \frac{2\eta u}{\pi}} \geq 0.5t} \frac{1}{\abs{s - \eta + \frac{2\eta iu}{\pi}}} \frac{du}{(\cosh u)^2} \\
    &< 4 e^{-\frac{\pi t}{40}} + 4/t \leq 2 e^{-1937},
\end{align*}
which completes the proof.
\end{proof}
In our proof, we require strong upper bounds for the number of zeros of $L(s,\chi)$ near the point $1+it$. Let $N_\chi(t,R)$ denote the number of zeros $\rho$ of $L(s,\chi)$ satisfying $|1+it-\rho| \leq R$. 
\begin{lemma}\label{lemma4.2}
Assume that there exist positive constants $A$ and $B$ such that \eqref{eq1.2} holds. Let $q \geq 3$, and let $\chi\pmod{q}$ be a Dirichlet character. If $1.04 (\log(t/100))^{-2/3} \leq R \leq 1/4$  and $t \geq \max\{e^{1938},q^{\frac{1}{100\,000}}\}$, then we have
\begin{align*}
    N_\chi(t,R) \leq 1.3478 R^{3 / 2} B \log t+0.49+\frac{\log (A+1)-\log R+1.8579R\log q + \frac{2}{3} \log \log t}{1.879}.
\end{align*}
\end{lemma}
\begin{proof}
Apply Lemma \ref{lemma4.1} with $s = 1+ 0.6421R + it$, $\eta = 2.5R$, and
\[
S = \{z \in \C: |1+it-z| \leq R,~\Re(z) \leq 1\}.
\]
(Note that with these choices, $\sigma - \eta \geq 1/2$.) By Lemma \ref{zetabound},
\begin{equation}\label{eq4.1}
    \begin{aligned}
        \Big|\frac{L^{\prime}}{L}(s+\eta+i v,\chi)\Big| & \leq \frac{1}{3.1421 R }, \\
        |L(s+\eta+i v,\chi)|^{-1} & \leq \zeta(1+3.1421 R) \leq 0.6+\frac{1}{3.1421 R} .
    \end{aligned}
\end{equation}
By \cite[(4.2)]{ford}, in the region $U=\{z: \Re(z) \geq 0.6421,|z-0.6421| \leq 1\}$, we have that
\begin{equation}\label{eq4.2}
    \Re \Big(\frac{\pi}{5} \cot \Big(\frac{\pi z}{5}\Big)\Big) \geq 0.3758.
\end{equation}
Thus, it follows from \eqref{eq4.1}, \eqref{eq4.2}, and Lemma \ref{lemma4.1} that
\begin{align*}
-\frac{1}{3.1421 R} &\leq-0.3758 \frac{N_\chi(t,R)}{R}+\frac{1}{5 R}\Big(1.8579R\log q + \frac{2}{3} \log \log t+ (1.8579 R)^{3 / 2} B \log t \\
&\qquad+\log(A+1) + \log \Big(0.6+\frac{1}{3.1421 R}\Big)\Big) + (1-\mathfrak{a})(1-\delta(\chi)) e^{-1937}.
\end{align*}
By the argument concluding the proof of \cite[Lemma 4.2]{ford}, we have that $\log(0.6 + \frac{1}{3.1421}R) \leq -\log R - 0.6735$, which completes the proof. 
\end{proof}
\begin{lemma}\label{lemma4.3}
Let $q \geq 3$, and let $\chi\pmod{q}$ be a Dirichlet character. If $t \geq \max\{e^{1938}, q^{\frac{1}{100\,000}}\}$ and $1.04 (\log(t/100))^{-2/3} < v \leq 1/4$, then
{\small
\begin{align*}
\sum_{L(\rho,\chi) = 0 \atop |1+i t-\rho| > v} \frac{1}{|1+i t-\rho|^{2}} &\leq (8.14467+5.3912 B(v^{-1 / 2}-2)) \log t-8.5 \log (A+1) \\
&+518.7+\frac{1}{v^{2}}\Big(\frac{\log(A+1)-\log v+\frac{2}{3} \log \log t}{1.879}+0.224\Big)-\frac{N_\chi(t,v)}{v^{2}} \\
&+ (\log q)(1.978 v^{-1}+0.23267).
\end{align*}
}%
\end{lemma}
\begin{proof}
Assume without loss of generality that $\chi$ is primitive. We also assume without loss of generality that $t \pm 1$ is not the ordinate of a zero. (If $\chi$ is imprimitive and $\chi'$ induces $\chi$, then $\chi'$ has the same nontrivial zeros as $\chi$ and smaller modulus. If $t \pm 1$ is the ordinate of a zero, apply the result to a sequence with $t'$ tending to $t$ and $v'$ tending to $v$ from above.) By \cite[Lemma 4.3]{fordzfr}, we may also assume that $\chi$ is nonprincipal. We partition the zeros with $|1+it-\rho| > v$ into the sets 
\[
\begin{aligned}
Z_{0} &=\{\rho:|\Im(\rho)| \leq 1\}, \\
Z_{1} &=\{\rho \notin Z_0 :|\Im(\rho)-t| \geq 1\}, \\
Z_{2} &=\Big\{\rho \notin Z_0 \cup Z_{1}:|1+i t-\rho| \geq \frac{1}{4} \text { and }|i t-\rho| \geq \frac{1}{4}\Big\}, \\
Z_{3} &=\{\rho: \rho \notin Z_0 \cup Z_1 \cup Z_2 \text { and }|1+i t-\rho| > v\} .
\end{aligned}
\]
For $k\in\{0,1,2,3\}$, we define
\begin{equation*}
    S_k = \sum_{\rho \in Z_k} \frac{1}{|1+it-\rho|^2}.
\end{equation*}
First, note that by \cite[Theorem 2.1]{mccurley} with $\eta = 0.5$, the number of zeros of $L(s,\chi)$ with $t \leq 1$ is at most $1.919\log q + 2.54$. Since $q \leq t^{100\,000}$ and $t \geq e^{1938}$, it follows that $|S_0| \leq e^{-1800}$.

Let $Q^\pm(T) = N^\pm(T,\chi) - \frac{T}{2\pi}\log \frac{qT}{2\pi e} - C_0(\chi)$. By Corollary \ref{corollary}, we have that
\begin{align*}
    |Q^\pm(T)| \leq 0.2297\log(qT) + 24.77.
\end{align*}
Write $C_1 = 0.2297$ and $C_2 = 24.77$. Observe that
\begin{align*}
   S_{1} \leq \int_{t+1}^{\infty} \frac{d N^+(u)}{(u-t)^{2}}+\int_{1}^{t-1} \frac{d N^+(u)}{(t-u)^{2}}+\int_{1}^{\infty} \frac{d N^-(u)}{(u+t)^{2}}=I_{1}+I_{2}+I_{3} .
\end{align*}
Since $dN^{\pm}(u) = \frac{1}{2\pi}\log q + \frac{1}{2\pi}\log(\frac{u}{2\pi}) + dQ^{\pm}(u)$ and $\log(t+x) \leq \log(t) + \frac{x}{t}$, we find that
\[
\begin{aligned}
I_{1} & \leq   \frac{1}{2 \pi} \int_{1}^{\infty} \frac{\log (t+x) + \log q -\log 2 \pi}{x^{2}} d x+|Q^+(t+1)|+2 \int_{1}^{\infty} \frac{|Q^+(t+x)|}{x^{3}} d x \\
&=\frac{\big(1+\frac{1}{t}\big) \log (1+t) + \log q -\log (2 \pi)}{2 \pi}+|Q^+(t+1)|+2 \int_{1}^{\infty} \frac{|Q^+(t+x)|}{x^{3}} d x \\
& \leq (2C_1 + 0.15916)\log (qt) +(2C_2-0.292507)+2 \int_{1}^{\infty} \frac{C_2 x / t}{x^{3}} d x \\
& \leq (2C_1 + 0.15916)\log(qt)+ (2C_2-0.2925) .
\end{aligned}
\]
Similarly, we find
\[
\begin{aligned}
I_{2} & \leq \frac{1}{2 \pi} \log \Big(\frac{qt}{2 \pi}\Big)+ \frac{|Q^+(1)|}{(t-1)^2} + 2 \max _{1 \leq u \leq t-1}|Q^+(u)| \\
& \leq (2C_1 + 0.159155)\log (qt)+(2C_2 - 0.292507),
\end{aligned}
\]
and since $q \leq t^{100\,000}$ and $t \geq e^{1938}$,
\[
I_{3} \leq \frac{1}{2 \pi} \int_{1}^{\infty} \frac{\log \big(\frac{u+t}{2 \pi}\big) + \log q}{(u+t)^{2}} d u+ \frac{\abs{Q^-(1)}}{(1+t)^2} +2 \int_{1}^{\infty} \frac{|Q^-(u)|}{(u+t)^{3}} d u \leq 0.000001 .
\]
Consequently, $S_1$ is upper bounded by
\begin{equation}\label{eq4.4}
    \begin{aligned}
    |S_1| &\leq (4C_1 + 0.31831)\log (qt) + (4C_2-0.58) \\
    &= 1.23711\log (qt)+ 98.5.
    \end{aligned}
\end{equation}

It follows from the proof of \cite[Lemma 4.3]{fordzfr}, \textit{mutatis mutandis}, that
\begin{align*}
S_2 \leq \frac{80N_2}{9},\qquad S_3 \leq \frac{80 N_{3}}{9}+\Big(\frac{80}{9}-\frac{1}{v^{2}}\Big) N_\chi(t,v)+2 \int_{v}^{1 / 4} \frac{N_\chi(t,u)}{u^{3}} d u.
\end{align*}
Furthermore, since $t \geq e^{1938}$, we have
\begin{equation}\label{eq4.5}
\begin{aligned}
N_{2}+N_{3} &=N^+(t+1,\chi)-N^+(t-1,\chi)-N_\chi(t,v) \\
& \leq (2C_1 + 0.31831)\log (qt) +(2C_2-0.585)-N_\chi(t,v) \\
&\leq 0.7771\log (qt)+ 49 - N_\chi(t,v).
\end{aligned}
\end{equation}
Lemma \ref{lemma4.2} implies that
\[
\begin{aligned}
&2 \int_{v}^{1 / 4} \frac{N_\chi(t,u)}{u^{3}} d u \leq\Big(\frac{\log(A+1) + \frac{2}{3} \log \log t}{1.879}+0.49\Big)(v^{-2}-16) \\
&\quad+5.3912 B(v^{-1 / 2}-2) \log t+\frac{1}{1.879}\Big(8-16 \log 4-\frac{1+2 \log v}{2 v^{2}}\Big) + 1.978(\log q)(v^{-1} -4).
\end{aligned}
\]
By \eqref{eq4.5}, it therefore follows that
\[
\begin{aligned}
S_{2}+S_{3} \leq &(6.907556+5.3912 B(v^{-1 / 2}-2)) \log t-8.5 \log (A+1) \\
&+420.2+\frac{1}{v^{2}}\Big(\frac{\log(A+1)-\log v+\frac{2}{3} \log \log t}{1.879}+0.224\Big)-\frac{N_\chi(t,v)}{v^{2}} \\
&+ (\log q)(1.978 v^{-1}-1.00444).
\end{aligned}
\]
Combining this with the upper bound for $S_1$ in \eqref{eq4.4} completes the proof.
\end{proof}
The next two lemmata correspond to \cite[Lemmata 4.5 and 4.6]{ford}, which in turn are based on \cite[Lemmata 5.1 and 5.2]{hb}.
\begin{lemma}\label{lemma4.5}
Let $0<\eta\leq \frac{3}{2}$.  Suppose that $f$ is a non-negative real function which has continuous derivative on $(0,\infty)$. Suppose that the Laplace transform $F(z)= \int_0^\infty f(y) e^{-zy} dy$ is absolutely convergent for $\Re(z) > 0$. Let $F_0(z) = F(z) - f(0)/z$, and suppose that there exists a positive constant $D$ possibly depending on $f$ such
\begin{align}\label{eq4.6}
 \mathrm{Re}(z)\geq 0,\qquad |z|\geq \eta,\qquad |F_0(z)| \leq \frac{D}{|z|^2}.
\end{align}
For $\Re(s) > 1$, $\Im(s) \geq 0$, and $\chi\pmod{q}$ a Dirichlet character with $q \geq 3$, we define
\[
K_\chi(s) :=\sum_{n=1}^{\infty} \Lambda(n) \chi(n) n^{-s} f(\log n),
\]
and $Q(q) = \log\log q + 0.66$.  We have that
\begin{align*}
&\Big|K_{\chi}(s)-\Big(-f(0) \frac{L^{\prime}(s,\chi)}{L(s,\chi)}-\sum_{L(\rho,\chi) = 0} F_{0}(s-\rho) - (1-\mathfrak{a}) (1-\delta(\chi)) F_0(s) +\delta(\chi)F_{0}(s-1)\Big)\Big|\\
&\leq D\Big(3.44 + \frac{1}{3}\log q + \frac{1}{3}\log(1 + \Im(s)) + \ms{I}(\chi)(f(0) + \max_{x\in (0,\infty)} f(x))Q(q)\Big) .
\end{align*}
\end{lemma}
\begin{proof}
We follow the proof of \cite[Lemma 4.5]{fordzfr}. Suppose that $\chi$ is primitive and that $q \geq 3$. Suppose $s = \sigma + it$ and $1 < \alpha < \sigma$. Define 
\begin{align*}
    I=\frac{1}{2 \pi i} \int_{\alpha-i \infty}^{\alpha+i \infty}-\frac{L^{\prime}(w,\chi)}{L(w,\chi)} F_{0}(s-w) d w.
\end{align*}
Since $-L'(w,\chi)/L(w,\chi) = \sum_n \chi(n) \Lambda(n) n^{-s}$, the sum converging uniformly on $\Re(w) = \alpha$, we may integrate term by term. Thus, with
\begin{align*}
    J_{n}=\frac{1}{2 \pi i} \int_{\alpha-i \infty}^{\alpha+i \infty} n^{-w} F_{0}(s-w) d w=\frac{n^{-s}}{2 \pi i} \int_{\sigma-\alpha-i \infty}^{\sigma-\alpha+i \infty} n^{u} F_{0}(u) d u,
\end{align*}
we have $I = \sum_n \Lambda(n) \chi(n) J_n$ where $J_n = n^{-s}(f(\log n) - f(0))$ via the proof of \cite[Lemma 4.5]{fordzfr}. Consequently, we have
\begin{equation*}
    I = K_\chi(s) + f(0) \frac{L'(s,\chi)}{L(s,\chi)}.
\end{equation*}
Moving the line of integration to $\Re(w) = -1/2$, we have
\begin{equation}\label{eq4.8}
    I=\frac{1}{2 \pi i} \int_{-1 / 2-i \infty}^{-1 / 2+i \infty}-\frac{L^{\prime}(w,\chi)}{L(w,\chi)} F_{0}(s-w) d w-\sum_{\rho} F_{0}(s-\rho) - (1-\mathfrak{a})(1-\delta(\chi)) F_0(s) + \delta(\chi)F_{0}(s-1).
\end{equation}
By \eqref{eq4.6}, \eqref{eq4.8} and Lemma \ref{lemma3.2}, the integral in \eqref{eq4.8} is bounded above by $\frac{D}{2\pi}I'$, where 
\[
\begin{aligned}
I^{\prime} & \leq \int_{-\infty}^{\infty} \frac{\log q + 8.21+\frac{1}{2} \log (1+(u/2)^2)}{9 / 4+(u-t)^{2}} d u \\
&= \frac{2\pi}{3}\log q +  5.48\pi +\frac{1}{3} \int_{-\infty}^{\infty} \frac{\log (1+(t/2 +3v/4)^{2})}{1+v^{2}} d v \\
& \leq \frac{2\pi}{3}\log q +  5.48\pi+\frac{1}{3} \int_{-\infty}^{\infty} \frac{\log (1+t^{2})+\log (1+v^2)}{1+v^{2}} d v \\
& \leq 2\pi\Big(\frac{1}{3}\log q + 2.74 + \log(2) +\frac{\log (1+t)}{3}\Big)  .
\end{aligned}
\]
This proves that when $\chi$ is primitive and $q \geq 3$,
\begin{align*}
    K_\chi(s) &=-f(0) \frac{L^{\prime}(s,\chi)}{L(s,\chi)}-\sum_{L(\rho,\chi) = 0} F_{0}(s-\rho) - (1-\mathfrak{a}) (1-\delta(\chi)) F_0(s) +\delta(\chi) F_{0}(s-1)\\
    &\quad+D\Big(3.44 + \frac{1}{3}\log q + \frac{1}{3}\log(1 + \Im(s))\Big).
\end{align*}
The same conclusion follows when $\chi$ is primitive and $q = 1$ by \cite[Lemma 4.5]{fordzfr}.

Now, let $\chi$ be possibly imprimitive with $q \geq 3$, and let $\chi'$ be the character inducing $\chi$. By Corollary \ref{last}, we have that
\begin{align*}
    |K_\chi(s)-K_{\chi'}(s)| &\leq \sum_{n \in \N \atop (n,q) > 1}     f(\log n) \frac{\Lambda(n)}{n} \leq  Q(q) \max_{x \in (0,\infty)} f(x),\\
     \Big|f(0)\frac{L'}{L}(s,\chi) - f(0)\frac{L'}{L}(s,\chi') \Big| &\leq f(0) \sum_{n\in \N \atop (n,q) > 1} \frac{\Lambda(n)}{n} \leq Q(q) f(0)  .
\end{align*}
Consequently, the desired result holds for all Dirichlet characters with $q \geq 3$.
\end{proof}
\begin{corollary}\label{lemma4.6}
Assume that there exist positive constants $A$ and $B$ such that \eqref{eq1.2} holds. Let $\chi\pmod{q}$ be a Dirichlet character with $q \geq 3$. Suppose that $1.92(\log(t/100))^{-2/3} < \eta < \frac{1}{2}$. Let $f$ be a compactly supported function satisfying \eqref{eq4.6} for some positive constant $D$. Suppose $s = 1+it$ with $t \geq \max\{q^{\frac{1}{100\,000}},e^{1938}\}$. Then, we have
\small
\[
\begin{aligned}
\Re(K_\chi(s)) &\leq-\sum_{L(\rho,\chi) = 0 \atop |1+i t-\rho| \leq \eta} \Re\Big\{F(s-\rho)+f(0)\Big(\frac{\pi}{2 \eta} \cot \Big(\frac{\pi(s-\rho)}{2 \eta}\Big)-\frac{1}{s-\rho}\Big)\Big\} \\
&\quad+\frac{f(0)}{2 \eta}\Big[\frac{2 \log \log t}{3}+B \eta^{3 / 2} \log t+\log (A+1) -\frac{1}{2} \int_{-\infty}^{\infty} \frac{\log \big|L\big(s+\eta+\frac{2 \eta u i}{\pi},\chi\big)\big|}{(\cosh u)^2} d u\Big] \\
&\quad+D\Big(3.5+ \frac{\log q}{3} + \frac{\log t}{3}+\sum_{L(\rho,\chi) = 0 \atop |1+i t-\rho| > \eta} \frac{1}{|1+i t-\rho|^{2}}\Big) \\
&\quad+ \ms{I}(\chi)(f(0) +\max_{x \in (0,\infty)} f(x)) Q(q) + \frac{f(0)}{2}\log q +f(0)(1-\delta(\chi))(1-\mathfrak{a})e^{-1937}.
\end{aligned}
\]
\normalsize
Moreover, $K_{\chi_0}(1) \leq F(0) + 1.8D$. 
\end{corollary}
\begin{proof}
The final assertion of the lemma follows immediately from \cite[Lemma 4.6]{fordzfr}. Assume $\sigma > 1$. Since $t \geq e^{1938}$, $\Re(F_0(s-1)) \leq \frac{D}{t} \leq 0.0001D$, and similarly $-\Re(F_0(s)) \leq 0.0001D$ by \eqref{eq4.6}. Furthermore, by \eqref{eq4.6},
\begin{align*}
    \sum_{L(\rho,\chi) = 0 \atop |1+i t-\rho|>\eta}|F_{0}(s-\rho)| \leq D \sum_{L(\rho,\chi)= 0 \atop |1+i t-\rho|>\eta} \frac{1}{|1+i t-\rho|^{2}}.
\end{align*}
By combining Lemma \ref{lemma4.1} (with $S = \{z: \Re(z) \leq 1,~|1+it-z| \leq \eta\}$) and Lemma \ref{lemma4.5}, we obtain the inequality  
\small
\[
\begin{aligned}
\Re(K_\chi(s)) &\leq-\sum_{L(\rho,\chi) = 0 \atop |1+i t-\rho| \leq \eta} \Re\Big\{F(s-\rho)+f(0)\Big(\frac{\pi}{2 \eta} \cot \Big(\frac{\pi(s-\rho)}{2 \eta}\Big)-\frac{1}{s-\rho}\Big)\Big\} \\
&\quad+\frac{f(0)}{2 \eta}\Big[\frac{2 \log \log t}{3}+B \eta^{3 / 2} \log t+\log (A+1) -\frac{1}{2} \int_{-\infty}^{\infty} \frac{\log \big|L\big(s+\eta+\frac{2 \eta u i}{\pi},\chi\big)\big|}{(\cosh u)^2} d u\Big] \\
&\quad+D\Big(3.5+ \frac{\log q}{3} + \frac{\log t}{3}+ \sum_{L(\rho,\chi) = 0 \atop |1+i t-\rho| > \eta} \frac{1}{|1+i t-\rho|^{2}}\Big) \\
&\quad+ \ms{I}(\chi)(f(0) +\max_{x \in (0,\infty)} f(x)) Q(q) + \frac{f(0)}{2}\log q +f(0)(1-\delta(\chi))(1-\mathfrak{a})e^{-1937}.
\end{aligned}
\]
\normalsize
As in \cite[(4.10)]{ford}, 
we may take $\sigma \to 1^+$, which proves the lemma.
\end{proof}
\section{A Trigonometric Inequality}
We use an inequality that plays the same role as the inequality $3 + 4\cos\theta + \cos(2\theta) \geq 0$ in Mertens' treatment of the proof of the prime number theorem due to Hadamard and de la Vall\'ee Poussin. For any real numbers $a_1,a_2$, there exist real numbers $b_0,\dots , b_4$ such that
\begin{equation}\label{eq5.1}
    \sum_{j=0}^{4} b_{j} \cos (j \theta)=8(a_{1}+\cos \theta)^{2}(a_{2}+\cos \theta)^{2} \geq 0 \quad(\theta \in \mathbb{R}).
\end{equation}
In particular, we have that
\begin{equation}\label{eq5.2}
\begin{aligned}
&b_{4}=1, \quad b_{3}=4(a_{1}+a_{2}), \quad b_{2}=4(1+a_{1}^{2}+a_{2}^{2}+4 a_{1} a_{2}) \\
&b_{1}=(a_{1}+a_{2})(12+16 a_{1} a_{2}), \quad b_{0}=b_{2}-1+8(a_{1} a_{2})^{2}.
\end{aligned}
\end{equation}
\begin{lemma}\label{lemma5.1}
Suppose that $a_1,a_2 \in \R$, and that $b_1,b_2,\dots,b_4$ are defined by \eqref{eq5.2}. We have
\begin{equation*}
    \int_{-\infty}^{\infty} \frac{1}{(\cosh u)^2} \sum_{j=1}^{4} b_{j} \log |L(\sigma+\eta+i j t_{1}+i u t_{2},\chi^j)| d u \geq-2 b_{0} \log \zeta(1+\eta).
\end{equation*}
\end{lemma}
\begin{proof}
Denote by $I$ the integral in the lemma. By \cite[(5.4)]{fordzfr},
\begin{equation*}
    U(y):=\int_{-\infty}^{\infty} \frac{e^{i y u}}{\cosh ^2 u} d u=\frac{\pi y}{\sinh (\pi y / 2)} \geq 0.
\end{equation*}

By \cite[(5.2)-(5.4)]{ford}, alongside 
\begin{equation*}
\log|L(s,\chi)| = \Re\Big(\sum_{\substack{p \\ m \geq 1}}\frac{1}{m}p^{-ms} (\chi(p))^m\Big),
\end{equation*}
we find that
\[
\begin{aligned}
I &=\sum_{p, m} \frac{1}{m} p^{-m(\sigma+\eta)} \Re\Big(\sum_{j=1}^{4} b_{j} p^{-i j m t_{1}} e^{imj \arg(\chi(p))} \int_{-\infty}^{\infty} \frac{p^{-i m u t_{2}}}{(\cosh u)^2} d u\Big) \\
&=\sum_{p, m} \frac{1}{m} p^{-m(\sigma+\eta)} U(m t_{2} \log p) \sum_{j=1}^{4} b_{j} \cos (j (m t_1 \log p -m\arg(\chi(p)))) \\
& \geq-b_{0} \sum_{p, m} \frac{1}{m} p^{-m(\sigma+\eta)} U(m t_{2} \log p) .
\end{aligned}
\]
Finally, since $U(y) \leq 2$, we obtain that $I \geq -2b_0\zeta(1+\eta)$.
\end{proof}
\section{The functions $f$, $F$, and $K$}
Let $f$ be a function with compact support, define $F$ and $K_\chi$ as in Lemma \ref{lemma4.5} and assume that \eqref{eq4.6} holds for some positive constant $D$. Let $a_1,a_2$ be real numbers and define $b_0,\dots,b_4$ by \eqref{eq5.2}. Let $b_5 = b_1 + \dots + b_4$, and write $L_1 = \log(4t+1)$, $L_2 = \log\log(4t+1)$. By \eqref{eq5.1},
\begin{equation}\label{eq6.2}
    \Re \sum_{j=0}^{4} b_{j} K_{\chi^j}(1+i j t)=\sum_{n=1}^{\infty}  \Lambda(n) n^{-1} f(\log n) \sum_{j=0}^{4} b_{j} \cos (j (t \log n - \arg(\chi(n)))) \geq 0.
\end{equation}
Assume that $\chi\pmod{q}$ is a primitive Dirichlet character with $q \geq 3$. Applying Lemma \ref{lemma4.6} with $s = 1+ijt$ ($1 \leq j \leq 4$), \eqref{eq5.1} and Lemma \ref{lemma5.1} with $t_2 = \frac{2\eta}{\pi}$,
\small
\begin{equation}\label{eq6.3}
\begin{aligned}
0 \leq &-\Re \sum_{\substack{1 \leq j \leq 4 \\ L(\rho,\chi^j) = 0 \\
|1+i j t-\rho| \leq \eta}}  b_j\Big( F(1+ijt-\rho)+f(0)\Big(\frac{\pi}{2 \eta} \cot \Big(\frac{\pi(1+ijt-\rho)}{2 \eta}\Big)-\frac{1}{1+ijt-\rho}\Big)\Big) \\
&\quad+\frac{f(0)}{2 \eta}\Big[b_5\Big(\frac{2L_2}{3}+B \eta^{3 / 2} L_1+\log (A+1)\Big)+ b_0 \log \zeta(1+\eta)\Big] + b_0 F(0)\\
&\quad+D\Big(b_5\Big(3.5+ \frac{\log q}{3} + \frac{L_1}{3} \Big)+ 1.8 b_0 +\sum_{1 \leq j \leq 4} b_j \sum_{\substack{L(\rho,\chi^j) = 0 \\ |1+ijt-\rho| > \eta}} \frac{1}{|1+i jt-\rho|^{2}}\Big) \\
&\quad+ (b_2+b_3+b_4)(f(0) +\max_{x \in (0,\infty)} f(x)) Q(q) + b_5 \frac{f(0)}{2}\log q + (b_2 + b_4) e^{-1937}.
\end{aligned}
\end{equation}
\normalsize
Note that since $\chi$ is primitive, $\ms{I}(\chi) = 0$ and hence $(f(0) +\max_{x \in (0,\infty)} f(x)) Q(q)$ is multiplied by a factor of $b_2+b_3+b_4$ (rather than $b_5 = b_1 + b_2 + b_3 + b_4$) above. Similarly, $1-\mathfrak{a}(\chi^j)$ can only be nonzero when $j$ is odd, and hence $e^{-1937}$ is multiplied by a factor of $b_2 + b_4$.

We now explicitly define the functions $f$, $F$ and $K$. Suppose $\gamma \geq e^{1938}$, $L(\beta+i\gamma,\chi) = 0$, and that there exists $0 < \lambda \leq 1-\beta$ such that 
\begin{equation}\label{eq6.1}
    L(\sigma+it,\chi^j) \neq 0 \quad(1-\lambda<\sigma \leq 1, \gamma-1 \leq t \leq 4 \gamma+1, 0 \leq j \leq 4).
\end{equation}
Let $\theta$ be the unique solution to 
\begin{equation}\label{eq6.4}
    (\sin \theta)^2=\frac{b_{1}}{b_{0}}(1-\theta \cot \theta), \quad 0<\theta<\pi / 2,
\end{equation}
and define the real function
\begin{equation}\label{eq6.5}
    g(u)= \begin{cases}(\cos (u \tan \theta)-\cos \theta) \sec ^{2} \theta & |u| \leq \frac{\theta}{\tan \theta}, \\ 0 & \text { else. }\end{cases}
\end{equation}
For $u \geq 0$, let 
\begin{equation*}
    w(u) = (g*g)(u) = \int_{-\infty}^\infty g(u-t) g(t)  dt,\qquad W(z)=\int_{0}^{\infty} e^{-z u} w(u) d u.
\end{equation*}
From the above definitions it follows (see \cite[(6.6)]{fordzfr}, \cite[Lemma 6.1]{hb}) that
\begin{equation}\label{eq6.6}
\begin{aligned}
W(0) &=2 \sec ^{2} \theta(1-\theta \cot \theta)^{2}, \\
W(-1) &=2 \tan ^{2} \theta+3-3 \theta(\tan \theta+\cot \theta), \\
w(0) &=\sec ^{2} \theta(\theta \tan \theta+3 \theta \cot \theta-3).
\end{aligned}
\end{equation}
Then, set
\begin{equation}\label{eq6.7}
    f(u)=\lambda e^{\lambda u} w(\lambda u) \quad(u \geq 0),
\end{equation}
and
\begin{equation}\label{eq6.8}
    F(z)=\int_{0}^{\infty} e^{-z u} f(u) d u=W\Big(\frac{z}{\lambda}-1\Big).
\end{equation}
\section{An Inequality for the Real Part of a Zero}
\begin{lemma}\label{lemma7.1}
Assume that there exist positive constants $A,B$ such that \eqref{eq1.2} holds. Suppose $\chi\pmod{q}$ is a primitive Dirichlet character with $q \geq 3$. Suppose that $L(\beta+i\gamma) = 0$, with $\gamma \geq \max\{q^{\frac{1}{100\,000}},e^{1938}\}$ and $1.92(\log(\gamma/100))^{-2/3} < \eta \leq 1/4 $. Suppose that $\lambda > 0$ satisfies \eqref{eq6.1}. Finally, suppose that
\begin{equation}\label{eq7.1}
    1-\beta \leq \eta / 2, \quad 0<\lambda \leq \min \Big(1-\beta, \frac{1}{198} \eta\Big).
\end{equation}
Write $L_1 = \log(4\gamma+1)$ and $L_2 = \log\log(4\gamma+1)$. Then, we have that
\begin{multline*}
    \frac{1}{\lambda}\Big[0.16521-0.1876\Big(\frac{1-\beta}{\lambda}-1\Big)\Big]\\
    \leq 1.471 \frac{1-\beta}{\eta^{2}}+\frac{1}{2 \eta}\Big[\frac{666550}{200211}\Big(\frac{2}{3} L_{2}+B \eta^{3 / 2} L_{1}+\log(A+1)\Big)+\log \zeta(1+\eta)\Big] \\
    + 3.5146\lambda\Big[(8.47801 + 5.392 B(\eta^{-\frac{1}{2}}-2)) L_{1}+\frac{\frac{\log ((A+1) / \eta)+\frac{2}{3} L_{2}}{1.879}+0.224}{\eta^{2}}\\
    + (1.978\eta^{-1} + 0.56601)\log q + 522.75\Big] \\
    +\frac{b_5}{2b_0} \log q + \frac{b_2+b_3+b_4}{b_0}Q(q) + \frac{b_2 + b_4}{b_0}e^{-1937} .
\end{multline*}
\end{lemma}
\begin{proof}
We take $a_1 = 0.225$ and $a_2 = 0.9$. Then, by \eqref{eq5.2}, we have that
\begin{equation*}
\begin{array}{ll}   
b_{0}=10.01055 & b_{3}=4.5 \\
b_{1}=17.14500 & b_{4}=1.0 \\
b_{2}=10.68250 & b_{5}=33.3275,
\end{array}
\end{equation*}
and by \eqref{eq6.4} and \eqref{eq6.6}, 
\begin{equation*}
    \theta=1.152214629976363048877 \ldots, \quad w(0)=6.82602968445295450905 \ldots.
\end{equation*}
Let 
\[
\begin{aligned}
&c_0=\frac{1}{\sin \theta (\cos \theta)^3}=16.2983216223932350562 \ldots \\
&c_1=(\theta-\sin \theta \cos \theta) (\tan \theta)^4=19.9352005926244107856 \ldots \\
&c_2=(\tan \theta)^3 (\sin \theta)^2=9.4813169452950521682 \ldots \\
&c_3=(\theta-\sin \theta \cos \theta) (\tan\theta)^2=3.945405755634895592 \ldots
\end{aligned}
\]
Let $R = 197$, let 
\begin{equation*}
H(R)=\frac{c_{0}(c_{2} \frac{(R+1)^{2}}{R^{3}}(e^{2 \theta / \tan \theta}+1)+\frac{c_{1}}{R^{2}}+c_{3})}{(1-\frac{\tan ^{2} \theta}{R^{2}})^{2}} ,
\end{equation*}
\begin{equation*}
    c_{4}=\frac{H(R)(R+1)^{2}}{R^{3} w(0)}+1+1 / R,
\end{equation*}
By \cite[(7.6)]{fordzfr} and the assumption $\eta \geq 198\lambda$, \eqref{eq4.6} holds with
\begin{equation}\label{eq7.6}
    D = c_4 \lambda f(0).
\end{equation}
Furthermore, a Mathematica calculation shows that $e^{x}w(x) \leq 7.23$, and thus that 
\begin{equation}\label{maxf}
    f(0)+\max_{x \in (0,\infty)} f(x) \leq f(0) + 7.23\lambda = f(0) + f(0) \frac{7.23}{w(0)} \leq 2.06 f(0)
\end{equation}
Let  
\begin{equation}\label{zc}
z = \frac{\pi}{2 \eta}(1+ijt-\rho), \qquad c = \frac{\pi \lambda}{2 \eta}. 
\end{equation}
Define 
\begin{equation}\label{vcrelation}
    V_{c}(z)=c w(0)\Big(\cot z-\frac{1}{z}\Big)+W\Big(\frac{z}{c}-1\Big).
\end{equation}
By \eqref{eq6.7} and \eqref{eq6.8},
\begin{equation*}
    F(1+i j t-\rho)+f(0)\Big(\frac{\pi}{2 \eta} \cot \Big(\frac{\pi(1+i j t-\rho)}{2 \eta}\Big)-\frac{1}{1+i j t-\rho}\Big)=V_{c}(z).
\end{equation*}
We aim to bound the contribution to the double sum appearing in \eqref{eq6.3} from all zeros besides $\beta + i\gamma$. By \cite[(7.7)]{fordzfr}, we have the lower bound
\begin{equation}\label{eq7.7}
    \Re(V_{c}(z)) \geq-c_{5} c^{2} w(0) \quad\Big(\Re(z) \geq c,|z| \leq \frac{\pi}{2}\Big),
\end{equation}
where
\begin{equation*}
    c_{5}=\frac{4}{\pi^{2}}\Big(1+\frac{(R+1)^{2} H(R)}{w(0) R^{3}}\Big)=\frac{4}{\pi^{2}}\Big(c_{4}-1 / R\Big) .
\end{equation*}
This condition is satisfied for all zeros in the double sum appearing in \eqref{eq6.3} (with $z,c$ as defined in \eqref{zc}) by the assumption on $\lambda$ in \eqref{eq6.1}. Thus, by \eqref{zc} (with $t = \gamma$), \eqref{vcrelation}, and \eqref{eq7.7}, we have that
\begin{equation}\label{thelastequation}
\begin{aligned}
&-\Re \sum_{\substack{1 \leq j \leq 4 \\ L(\rho,\chi^j) = 0 \\
|1+ij\gamma-\rho| \leq \eta}} b_{j}\Big(F(1+ij\gamma-\rho)+f(0)\Big(\frac{\pi}{2 \eta} \cot \Big(\frac{\pi(1+ij\gamma-\rho)}{2 \eta}\Big)-\frac{1}{1+ij\gamma-\rho}\Big)\Big) \\
&\quad \leq-b_{1} V_{c}\Big(\frac{\pi}{2 \eta}(1-\beta)\Big)+c_{5} c^{2} w(0) \sum_{j=1}^{4} b_{j} N_{\chi^j}(j \gamma, \eta) .
\end{aligned}
\end{equation}
Write $L_1 = \log(4\gamma+1)$ and $L_2 = \log\log(4\gamma+1)$. Combining  Lemma \ref{lemma4.3} (with $t = \gamma$), \eqref{eq6.3}, \eqref{eq6.7}, \eqref{eq7.6}, \eqref{maxf}, and \eqref{thelastequation}, we find that
{\small
\begin{align}
\label{eq7.9}
&b_{0} F(0)-b_{1} V_{c}\Big(\frac{\pi}{2 \eta}(1-\beta)\Big)+\frac{\lambda f(0)}{\eta^{2}}\Big(\frac{\pi^{2} c_{5}}{4}-c_{4}\Big) \sum_{j=1}^{4} b_{j} N(j \gamma, \eta) \notag\\
&+\frac{f(0)}{2 \eta}\Big[b_5\Big(\frac{2L_2}{3}+B \eta^{3 / 2} L_1+\log (A+1)\Big)+ b_0 \log \zeta(1+\eta)\Big] + b_0 F(0)\notag\\
&+c_4\lambda f(0)b_5\Big[\Big(3.5+ \frac{\log q}{3} + \frac{L_1}{3}\Big)+ 1.8 \frac{b_0}{b_5} + \Big(8.14467+5.3912 B(v^{-1 / 2}-2)\Big) \log t \\
&+518.7-8.5\log(A+1)+\frac{1}{\eta^{2}}\Big(\frac{\log(A+1)-\log \eta+\frac{2}{3} L_{2}}{1.879}+0.224\Big) + (\log q) (1.978 v^{-1}+ 0.23267)\Big] \notag\\
&+(b_2 + b_3 + b_4) f(0)Q(q) + b_5 \frac{f(0)}{2}\log q + (b_2 + b_4) e^{-1937}\geq 0.\notag
\end{align}}%
Since $c_5 < \frac{4}{\pi^2}c_4$, we may drop the summation term in the first line above. Furthermore, the proof of \cite[Lemma 7.1]{ford} shows that
\begin{equation}\label{eq7.13}
\begin{aligned}
F(0)-\frac{b_1}{b_0} V_c\left(\frac{\pi}{2 \eta}(1-\beta)\right) & \leq 0.348 f(0) \frac{\pi^2}{4 \eta^2} \frac{b_1}{b_0}(1-\beta) \\
&+\frac{f(0)}{\lambda}\left(-\cos ^2 \theta+\frac{0.7475 b_1}{b_0 w(0)}\left(\frac{1-\beta}{\lambda}-1\right)\right) .
\end{aligned}
\end{equation}
Dividing both sides of the equation by $b_0f(0)$, employing the numerical values of $b_0$, $b_1$, $b_5$ and $\theta$, the inequalities $H(R) \leq 66.3307$ and $c_4 \leq 1.055656$ (for $R = 197$), and utilizing \eqref{eq7.13} in \eqref{eq7.9} gives the desired result.
\end{proof}

\section{Proof of Theorem \ref{theorem2}}

Recall the notation and hypotheses of Lemma \ref{lemma7.1}. In this section, we choose the free parameters $\lambda$ and $\eta$ appearing in Lemma \ref{lemma7.1} in order to complete the proof of Theorem \ref{theorem2}. The optimal choices are given in \eqref{eq8.2} and \eqref{eq8.3}. Throughout this section, we assume without loss of generality that $\chi$ is primitive. 

\subsection{Final choices of $\lambda$ and $\eta$}

Suppose that $T_0$ satisfies the hypotheses of Theorem \ref{theorem2}. If $1-\beta \geq ((10.082 + 1.607/\log\log T_0) \log q + 9.791\log\log q )^{-1}$ for all zeros $\beta+i\gamma$ of $L(s,\chi)$ satisfying $\gamma \geq T_0$, then Theorem \ref{theorem2} is immediate. We therefore assume that $1-\beta < ((10.082 + 1.607/\log\log T_0) \log q + 9.791\log\log q)^{-1}$ and $\gamma \geq T_0$ for some zeros of $L(s,\chi)$. Define
\begin{equation}\label{eq8.1}
\begin{aligned}
     &Z(\beta, \gamma)= \frac{(1-\beta) B^{2/3} (\log \gamma)^{2/3} (\log\log \gamma)^{1/3}}{1-(1-\beta)\big(\big(10.082 + \frac{1.607}{\log\log T_0}\big) \log q + \log\log q\big)} \\
     &S = \Big\{\beta + i \gamma: \exists \chi\pmod{q} ,~ L(\beta+i \gamma,\chi)=0 ,~ \gamma \geq T_{0} , \\
     \\&\qquad\qquad 1-\beta \leq \Big(\Big(10.082 + \frac{1.607}{\log\log T_0}\Big) \log q + 9.791\log\log q \Big)^{-1} \Big\}, \\
     &M=\inf _{\beta+i\gamma \in S} Z(\beta, \gamma).
\end{aligned}
\end{equation}

By Corollary \ref{inexplicitcorollary}, we have that $M>0$.  To prove Theorem \ref{theorem2}, it suffices to prove that $M\geq M_1$.  Suppose to the contrary that
\begin{equation}\label{contradiction}
M < M_1 \leq 0.055071.
\end{equation}
Under the hypothesis of \eqref{contradiction}, there exists a zero $\beta + i\gamma$ of $L(s,\chi)$ with $\gamma \geq T_0$, $1-\beta \leq (10.082 + \frac{1.607}{\log\log T_0})\log q + 9.791 \log\log q$, and
\[
Z(\beta, t) \in[M, M(1+\delta)], \quad \delta=\min \Big(\frac{10^{-100}}{\log T_{0}}, \frac{M_{1}-M}{2 M}\Big).
\]
By \eqref{mccurley}, we necessarily must have $\gamma \geq q^{\frac{1}{100\,000}}$, since 
\begin{equation*}
     \frac{1}{9.64590880\log(q\gamma)} \leq 1-\beta \leq \frac{1}{10.082 \log q}.
\end{equation*}
Let $L_1 = \log(4\gamma+1)$, and $L_2 = \log\log(4\gamma+1)$. By \eqref{eq8.1}, \eqref{eq6.1} holds with 
\begin{equation}\label{eq8.2}
    \lambda = \frac{1}{(10.082 + \frac{1.607}{\log\log T_0}) \log q + 9.791\log\log q +\frac{1}{M}B^{2/3} L_1^{2/3} L_2^{1/3}}.
\end{equation}
Define
\begin{equation}\label{eq8.3}
    E=\Big(\frac{4(1+b_{0} / b_{5})}{3}\Big)^{\frac{2}{3}}=\Big(\frac{1733522}{999825}\Big)^{\frac{2}{3}}, \quad \eta=E B^{-\frac{2}{3}}\Big(\frac{L_{2}}{L_{1}}\Big)^{\frac{2}{3}}.
\end{equation}

We now verify that our choices of $\lambda$ and $\eta$ satisfy the hypotheses of Lemma \ref{lemma7.1}. The lower bound $\frac{\log T_0}{\log\log T_0} \geq \frac{1139}{B}$ guarantees that $\eta \leq 0.014$, and the bound $T_0 \geq \exp(1938)$ ensures that $\lambda \leq 0.055071(B L_{1})^{-2/3} L_{2}^{-1/3} \leq \frac{\eta}{198}$.  Furthermore, the hypotheses $T_0 \geq e^{1938}$ and $B \leq 4.45$ imply that $\eta > 1.92 (\log(\gamma/100))^{-2/3}$. Finally, the inequalities $T_0 \geq e^{1938}$ and $M_1 \leq 0.055071$ ensure that the remaining hypotheses of Lemma \ref{lemma7.1} are met.    

\subsection{Proof of Theorem 3.1}
In order to contradict \eqref{contradiction}, we apply Lemma \ref{lemma7.1} with $\lambda$ and $\eta$ as defined in \eqref{eq8.2} and \eqref{eq8.3}, respectively.  With these choices, we will prove that
{\small
\begin{equation}\label{firstequation}
\begin{aligned}
&\frac{1}{2 \eta}\Big[\frac{666550}{200211}\Big(\frac{2}{3} L_{2}+B \eta^{3 / 2} L_{1}+\log(A+1)\Big)+\log \zeta(1+\eta)\Big], \\
&\leq 2.99968(B L_{1})^{\frac{2}{3}} L_{2}^{\frac{1}{3}} + \Big(\frac{B L_{1}}{L_{2}}\Big)^{\frac{2}{3}}\Big(1.1534 \log(A+1)-0.124193+0.23096 \log (B / L_{2})\Big),
\end{aligned}
\end{equation}
}
that
\begin{multline}\label{secondequation}
3.5146\lambda\Big[(8.47801 + 5.392 B(\eta^{-\frac{1}{2}}-2)) L_{1}+\frac{\frac{\log (A / \eta)+\frac{2}{3} L_{2}}{1.879}+0.224}{\eta^{2}} \\
+ (1.978\eta^{-1} + 0.56601)\log q + 522.75\Big] \\
\leq \Big(\frac{B L_{1}}{L_{2}}\Big)^{\frac{2}{3}}\Big[\frac{1.694-2.087 B}{B^{4 / 3}}\Big(\frac{L_{2}}{L_{1}}\Big)^{\frac{1}{3}}+0.006534\log(A+1)+0.0043556 \log \Big(\frac{B}{L_{2}}\Big)+0.954863\Big] \\
+ \frac{0.265271\log q}{\log\log T_0},
\end{multline}
and that
\begin{equation}\label{thirdequation}
    \begin{aligned}
        &\frac{b_5}{2b_0} \log q + \frac{b_2+b_3+b_4}{b_0}Q(q) + \frac{b_2 + b_4}{b_0}e^{-1937} \\
        &\leq 1.66462\log q + 1.6166\log\log q + 0.009783 B^{2/3} \Big(\frac{L_1}{L_2} \Big)^{2/3}.
    \end{aligned}
\end{equation}
in subsections \ref{subsection1}, \ref{subsection2}, and \ref{subsection3} respectively. Assuming these inequalities for now, we proceed with our proof. Observe that since $M \leq M_1 \leq 0.055071$ and $T_0 \geq e^{1938}$,
\begin{equation}\label{lastequation}
    1.471 \frac{1-\beta}{\eta^{2}} \leq 0.0389 B^{2 / 3} L_{1}^{2 / 3} L_{2}^{-5 / 3} \leq 0.00514 B^{2 / 3}\Big(\frac{L_{1}}{L_{2}}\Big)^{2 / 3} .
\end{equation}
Next, note that by the lower bound on $T_0$,
\begin{equation}\label{eq8.4}
\begin{aligned}
\frac{1-\beta}{\lambda}-1 &\leq \frac{(10.082 + \frac{1.607}{\log\log T_0} ) \log q + 9.791 \log\log q + \frac{1}{M}B^{2/3} L_1^{2/3} L_2^{1/3}}{(10.082 + \frac{1.607}{\log\log T_0} ) \log q + 9.791\log\log q + \frac{1}{M(1+\delta)}B^{2/3}(\log \gamma)^{2/3} (\log\log \gamma)^{1/3}} \\
&\leq \frac{0.98525}{\log T_{0}}.
\end{aligned}
\end{equation}
Applying \eqref{firstequation}-\eqref{lastequation} in Lemma \ref{lemma7.1} and gathering terms, we obtain the inequality
\begin{align*}
    \frac{1}{\lambda}\Big(0.16521-\frac{0.184833}{\log T_0}\Big) &\leq (B L_{1})^{\frac{2}{3}} L_{2}^{\frac{1}{3}}\Big(2.99968+\frac{X(\gamma)}{L_{2}}\Big)  \\
    &\qquad+ 1.66462\log q + 1.6166\log\log q + \frac{0.265271 \log q}{\log\log T_0}.
\end{align*}

Rearranging and using the definition of $\lambda$ from \eqref{eq8.2}, we find that
\begin{equation}\label{lambdaineq}
\begin{aligned}
     \lambda &= \frac{1}{\frac{1}{M}(B L_{1})^{\frac{2}{3}}L_2^{1/3} + (10.082 + \frac{1.607}{\log\log T_0}) \log q + 9.791 \log\log q} \\
     &\geq  \frac{0.16521 - \frac{0.184833}{\log T_0}}{(B L_{1})^{\frac{2}{3}} L_{2}^{\frac{1}{3}}(2.99968+\frac{X(\gamma)}{L_{2}}) + \big(1.66462 + \frac{0.265271}{\log\log T_0} \big)\log q + 1.6166\log\log q  }.
\end{aligned}
\end{equation}
Note that since $0.16521 - \frac{0.184833}{\log T_0} \geq 0.1651146$ for $T_0 \geq e^{1938}$, we have that
\begin{equation}\label{crossmult}
\begin{aligned}
&\Big((10.082 + \frac{1.607}{\log\log T_0})(\log q) + 9.791\log\log q\Big)\Big(0.16521 - \frac{0.184833}{\log T_0}\Big) \\
&> \Big(1.66462 + \frac{0.265271}{\log\log T_0} \Big) \log q + 1.6166\log\log q .
\end{aligned}
\end{equation}
By utilizing \eqref{lambdaineq} and \eqref{crossmult}, we obtain the inequality
\begin{align*}
    &\frac{0.16521 - \frac{0.184833}{\log T_0}}{M}(B L_{1})^{\frac{2}{3}} L_2^{1/3} + \Big(1.66462 + \frac{0.265271}{\log\log T_0}\Big)\log q + 1.6166\log\log q   \\
    &\qquad< (B L_{1})^{\frac{2}{3}} L_{2}^{\frac{1}{3}}\Big(2.99968+\frac{X(\gamma)}{L_{2}}\Big) + \Big(1.66462 + \frac{0.265271}{\log\log T_0}\Big)\log q + 1.6166\log\log q ,
\end{align*}
which implies that
\begin{equation*}
    M > \frac{0.16521 - \frac{0.184833}{\log T_0}}{2.99968+\frac{X(\gamma)}{\log\log \gamma}} \geq M_1.
\end{equation*}
This contradicts \eqref{contradiction}, as desired, and our proof of Theorem \ref{theorem2} is complete.

\subsection{Proof of \eqref{firstequation}}\label{subsection1}
By Lemma \ref{zetabound}, we have that
\begin{equation}\label{eq8.5}
\log \zeta(1+\eta) \leq \log (1 / \eta+0.6) \leq \log (1 / \eta)+0.0084.
\end{equation}
Since $\log(1/\eta) \sim \frac{2}{3}L_2$, it follows that the main term involving $t$ in \eqref{firstequation} (which in fact gives the main term in the final result) is bounded above by
\begin{equation}\label{eq8.6}
\frac{1}{2 \eta}\Big[\frac{b_{5}}{b_{0}}\Big(\frac{2 L_{2}}{3}+B \eta^{\frac{3}{2}} L_{1}\Big)+\frac{2 L_{2}}{3}\Big] =\frac{b_{5}}{b_{0}}\Big(1+\frac{b_{0}}{b_{5}}\Big)^{\frac{1}{3}}\Big(\frac{3 B}{4}\Big)^{\frac{2}{3}} L_{1}^{\frac{2}{3}} L_{2}^{\frac{1}{3}}  \leq 2.99968(B L_{1})^{\frac{2}{3}} L_{2}^{\frac{1}{3}}.
\end{equation}
Note that $\log\log T_0 \geq 7.569 > 4.45 \geq B$, meaning that $\log(B/L_2) < 0$. Hence, the remaining terms in \eqref{firstequation} give
\[
        \begin{aligned}
&\leq \frac{1}{2 \eta}\Big[\frac{b_{5}}{b_{0}} \log(A+1)-\log E+\frac{2}{3} \log(B / L_{2})+0.0084\Big] \\
&\leq \frac{B^{\frac{2}{3}}}{2 E}\Big(\frac{L_{1}}{L_{2}}\Big)^{\frac{2}{3}}\Big[\frac{b_5}{b_0} \log(A+1)-0.35848+\frac{2}{3} \log(B / L_{2})\Big] \\
&\leq\Big(\frac{B L_{1}}{L_{2}}\Big)^{\frac{2}{3}}\Big(1.1534 \log(A+1)-0.124193+0.23096 \log (B / L_{2})\Big).
\end{aligned}
\]
\subsection{Proof of \eqref{secondequation}}\label{subsection2}
We next consider the terms in \eqref{secondequation} not involving $q$. Since $522.75 \leq 0.26974 \log \gamma$ and $T_0 \geq e^{1938}$, these terms contribute at most
\small
\begin{equation*}
    \begin{aligned}
&0.193553 L_{1}^{-\frac{2}{3}} L_{2}^{-\frac{1}{3}} B^{-\frac{2}{3}}\Big[\Big(8.74775+\frac{5.392 B^{\frac{4}{3}}}{\sqrt{E}}\Big(\frac{L_{1}}{L_{2}}\Big)^{\frac{1}{3}}-10.784 B\Big) L_{1}\Big. \\
&\Big.\quad+\frac{B^{\frac{4}{3}}}{E^{2}}\Big(\frac{L_{1}}{L_{2}}\Big)^{\frac{4}{3}}\Big(\frac{\log(A+1)+\frac{4}{3} L_{2}+\frac{2}{3} \log(B / L_{2})-\log E}{1.879}+0.224\Big)\Big] \\
&\leq\Big(\frac{B L_{1}}{L_{2}}\Big)^{\frac{2}{3}}\Big[0.9347+\frac{1.694-2.087 B}{B^{\frac{4}{3}}}\Big(\frac{L_{2}}{L_{1}}\Big)^{\frac{1}{3}}+\frac{0.049454}{L_{2}}\Big(\log(A+1)+\frac{2}{3} \log(B / L_{2})+0.05401\Big)\Big] \\
&\leq\Big(\frac{B L_{1}}{L_{2}}\Big)^{\frac{2}{3}}\Big[\frac{1.694-2.087 B}{B^{4 / 3}}\Big(\frac{L_{2}}{L_{1}}\Big)^{\frac{1}{3}}+0.006534\log(A+1)+0.0065334 \cdot \frac{2}{3} \log \Big(\frac{B}{L_{2}}\Big)+0.935053\Big].
    \end{aligned}
\end{equation*}
\normalsize
Utilizing the assumption that $\frac{\log T_0}{\log \log T_0} \geq \frac{1139}{B}$, we find that the terms in \eqref{secondequation} involving $q$ are bounded by
\begin{equation}\label{thirdlineq}
    \begin{aligned}
        &3.5146\lambda \Big(  \frac{1.978}{\eta} + 0.56601\Big) \log q \\
        &\leq \frac{3.5146}{10.082  \log q + \frac{1}{0.055071} B^{2/3} L_1^{2/3} L_2^{1/3} } \cdot \log q \cdot \Big(1.978\frac{B^{2/3}}{E} \cdot \frac{L_1^{2/3}}{L_2^{2/3}} + 0.56601 \Big)\\
        &\leq \frac{6.9519}{E}  \frac{ B^{2/3}L_1^{2/3}}{\frac{1}{0.055071}B^{2/3} L_1^{2/3} L_2} \log q + 0.197312 \cdot 1139^{-2/3} \cdot B^{2/3} \Big(\frac{L_1}{L_2} \Big)^{2/3} \\
        &\leq \frac{1}{L_2} 0.055071\frac{6.9519}{E} \log q + 0.00181  B^{2/3} \Big(\frac{L_1}{L_2} \Big)^{2/3}  \\
        &\leq \frac{0.265271\log q}{\log\log T_0} + 0.00181  B^{2/3} \Big(\frac{L_1}{L_2} \Big)^{2/3}.
    \end{aligned}
\end{equation}
\subsection{Proof of \eqref{thirdequation}}\label{subsection3}
Utilizing the numerical values of $b_0$, $b_2$, $b_3$ and $b_4$, along with the lower bound $\frac{\log T_0}{\log\log T_0} \geq \frac{1139}{B}$ , we find that
\begin{align*}
    \frac{b_2 + b_3 + b_4}{b_0}Q(q) &\leq 1.6166(\log\log q + 0.66) \leq 1.6166\log\log q + 0.009783 B^{2/3} \Big(\frac{L_1}{L_2} \Big)^{2/3}.
\end{align*}
Finally, we utilize the numerical values of $b_0, \dots, b_5$ to bound $\frac{b_5}{2b_0} \log q + \frac{b_2 + b_4}{b_0}e^{-1937}$ above by $1.66462\log q$.

\appendix
\section{Explicit estimate for $\sum_{p \mid q} \frac{1}{p} .$}
\begin{lemma}\label{a1}
Let
\begin{equation*}
    E=-\gamma_{\Q}-\sum_{n=2}^{\infty} \sum_p\frac{\log p}{p^n}.
\end{equation*}
If $q \geq 2310$, 
then $\sum_{p \mid q} \frac{\log p}{p}
 <
\log\log q + E + 0.1313$.
\end{lemma}
\begin{proof}
Let $P_x = \prod_{p \leq x} p$ (a primorial).  Choose $x$ so that $P_x$ is the greatest primorial less than or equal to $q$. Note that $\log(3)/3 > \log(2)/2 > \log(5)/5$, and that for $n \geq 3$ the function $\log(n)/n$ is decreasing. Since $q \geq 2310 $, this observation implies that
\begin{equation*}
    \sum_{p \mid q} \frac{\log p    }{p} \leq \sum_{p \leq x} \frac{\log p}{p} = \sum_{p \mid P_x} \frac{\log p}{p}.
\end{equation*}
It therefore follows that any upper bound for $\sum_{p \mid q} \frac{\log p}{p}$ which is increasing in $q$ and applies when $q \geq 2310$ is primorial in fact holds for all $q \geq 2310$. We thus assume without loss of generality that there exists an $x$ such that $q = P_x$. Since $P_x=e^{\theta(x)}$, and it follows from  \cite[(3.14)]{rs} that $e^{\theta(x)} > e^{x(1 - \frac{1}{2\log(2310)} )}$, we have that $x < \frac{\log q}{1 - \frac{1}{2\log(2310)}}$. By \cite[(2.8), (3.22)]{rs}, we have that
\begin{equation*}
    \begin{aligned}
        \sum_{p \mid q} \frac{\log p}{p} &= \sum_{p \leq x} \frac{\log p}{p} < \log x + E + 0.06456.
    \end{aligned}
\end{equation*}
The desired result follows in our range of $x$. 
\end{proof}
From the preceding lemma, we immediately obtain the following corollary.
\begin{corollary}\label{last}
If $q \geq 3$, then $\sum_{(n,q) > 1} \frac{\Lambda(n)}{n} < \log\log q + 0.66$.
\end{corollary}
\begin{proof}
For $3 \leq q \leq 2310$, the statement can be verified by computer. For $q \geq 2310$, observe that by Proposition \ref{a1},
\[
    \sum_{(n,q) > 1} \frac{\Lambda(n)}{n} \leq \sum_{p \mid q}\frac{\log p}{p} + \sum_{n = 2}^\infty \sum_{p} \frac{\log p}{p^n} = \sum_{p \mid q} \frac{\log p}{p} -\gamma_{\Q} - E< \log\log q + 0.1313-\gamma_{\mathbb{Q}}.
\]
\end{proof}
\section{Preliminary Vinogradov--Korobov for Dirichlet $L$-functions}
\label{appB}
\begin{theorem}\label{inexplicit}
Let $T_0 \geq e^{10\,650}$. Assume that there exist positive constants $A$ and $B\leq 4.45$ such that \eqref{eq1.2} holds, that $\frac{\log T_0}{\log \log T_0} \geq \frac{5110.6}{B}$, and that $\log\log T_0 \geq \frac{183}{B^2}$. For any Dirichlet character $\chi\pmod{q}$ with $q \geq 3$, $L(\sigma+it,\chi) \neq 0$ when $t \geq T_0$ and
\begin{equation}\label{weak}
    \sigma>1-\frac{1}{18\log q+c B^{2/3} (\log t)^{2 / 3}(\log \log t)^{1 / 3}},
\end{equation}
where 
\begin{equation*}
    c = 31.76 + \max\Big\{\max_{t \geq T_0} \frac{-2.89\log_3 t + 14.44\log(A+1) + 3.59}{\log \log t}, 0 \Big\}.
\end{equation*}
In particular, if $\log_3 t \geq 5\log(A+1) + 1.25$, then $c = 31.76$.
\end{theorem}
\begin{proof}
Suppose for sake of contradiction that there existed a zero $\beta + i\gamma$ of $L(s,\chi)$ with
\begin{equation}\label{betalazy}
    \beta \geq 1-\frac{1}{18\log q+c B^{2/3}(\log \gamma)^{2 / 3}(\log \log \gamma)^{1 / 3}},\qquad t\geq 3,
\end{equation}
and $\gamma \geq T_0$. By the zero-free region in \eqref{mccurley}, we must have $\gamma \geq q^{\frac{1}{100\,000}}$, since
\begin{equation*}
\frac{1}{9.64590880\log(q\gamma)} \leq 1-\beta \leq \frac{1}{18\log q}.
\end{equation*}
Set $a_1 = 0.225$ and $a_0=9$ in \eqref{eq5.2}. Let $C = 4/3$ and
\begin{equation}\label{etalazy}
    \eta = \Big(\frac{C}{B}\frac{\log \log \gamma}{\log \gamma}\Big)^{2/3} .
\end{equation}
By the assumption $\frac{\log\log T_0}{\log T_0} \geq \frac{5110.6}{B}$, we have that $(\frac{4\log \log \gamma}{3B \log \gamma})^{2/3} \leq (\frac{4\log\log T_0}{3B\log T_0})^{2/3} < 1/2$, i.e., that $\eta < 1/2$.  We apply Lemma \ref{lemma4.1} with $s = \sigma+ij\gamma$ for $1 \leq j \leq 4$, with $S = \{\beta+i\gamma\}$ for $j=1$ and $S = \emptyset$ for $j \neq 1$, and with some $\sigma$ (to be specified below) such that $1 \leq \sigma \leq 1 + \eta - (\log(\gamma/100))^{-2/3}$. We also apply \eqref{eq5.1} and Lemma \ref{lemma5.1} with $t = \frac{2\eta}{\pi}$, to obtain that
\begin{equation}\label{firstpart}
\begin{aligned}
    0 &\leq -b_1 \Re\Big( \frac{\pi}{2 \eta} \cot \Big(\frac{\pi(\sigma-\beta)}{2 \eta}\Big)\Big) \\
    &\quad+\frac{b_5}{2 \eta}\Big((1-\sigma+\eta)\log q +  \frac{2}{3} \log \log \gamma+B(1-\sigma+\eta)^{3 / 2} \log \gamma+\log(A+1)\Big) + \frac{b_0}{\sigma-1} \\
    &\quad-\frac{1}{4 \eta} \sum_{j=1}^4 b_j \int_{-\infty}^{\infty} \frac{\log |L(\sigma+\eta+it +2 \eta i u / \pi,\chi^j)|}{(\cosh u)^2} d u + (b_2 + b_4)e^{-1937}  \\
    &\leq  -b_1 \Re\Big( \frac{\pi}{2 \eta} \cot \Big(\frac{\pi(\sigma-\beta)}{2 \eta}\Big)\Big) +\frac{b_5}{2 \eta}\Big(\frac{2}{3} \log \log \gamma+B\eta^{3 / 2} \log \gamma+\log(A+1)\Big) \\
    &\quad+ \frac{b_0}{\sigma-1}+\frac{b_0}{2 \eta} \log(\zeta(1+\eta)) + \frac{b_5}{2} \log q + (b_2 + b_4)e^{-1937}.
\end{aligned}
\end{equation}
We now set $\sigma = 1+ F(1-\beta)$ with $F = 3.238$. Observe that since $\log\log T_0 \geq \frac{183}{B^2}$ and $\gamma \geq T_0 \geq e^{10\,650}$,
\begin{align*}
    1+\eta-\sigma = \eta - F(1-\beta) &\geq \Big(\frac{C}{B}\frac{\log \log \gamma}{\log \gamma}\Big)^{2/3} - \frac{F}{31.76 B^{2/3} (\log \gamma)^{2/3} (\log \log \gamma)^{1/3}} \\
    &\geq 1.92 (\log(\gamma/100))^{-2/3},
\end{align*}
so indeed we have $1 \leq \sigma \leq 1 + \eta - 1.92 (\log(\gamma/100))^{-2/3}$. Furthermore, \eqref{betalazy} and \eqref{etalazy} imply that $\frac{\sigma-\beta}{\eta} = (F+1)\frac{1-\beta}{\eta} \leq 0.012$. Since $\cot(x) \geq 0.99988x^{-1}$ when $0 \leq x \leq 0.012 \frac{\pi}{2}$, we find that
\begin{align*}
    -b_1 \Re\Big( \frac{\pi}{2 \eta} \cot \Big(\frac{\pi(\sigma-\beta)}{2 \eta}\Big)\Big) \leq -0.99988\frac{b_1}{\sigma-\beta}.
\end{align*}
Hence, by the numerical values of $b_0$, $b_5$, and $F$, we have that
\begin{equation}\label{1minusb}
\begin{aligned}
     &-b_1 \Re\Big( \frac{\pi}{2 \eta} \cot \Big(\frac{\pi(\sigma-\beta)}{2 \eta}\Big)\Big) + \frac{b_0}{\sigma - 1} \leq -0.99988 \frac{b_1}{\sigma-\beta} + \frac{b_0}{F(1-\beta)} \\ &\qquad\leq \Big(\frac{10.01055}{F}-0.99988\frac{17.14500}{F+1}\Big)\frac{1}{1-\beta} \leq -0.953\frac{1}{1-\beta} .
\end{aligned}
\end{equation}
Combining \eqref{1minusb} with~\eqref{firstpart}, we obtain the inequality
\begin{equation}\label{lazykey}
\begin{aligned}
    0.953\frac{1}{1-\beta} &\leq \frac{b_5}{2 \eta}\Big(\frac{2}{3} \log \log \gamma+B\eta^{3 / 2} \log \gamma+\log(A+1)\Big) \\
    &\quad+\frac{b_0}{2 \eta} \log(\zeta(1+\eta)) + \frac{b_5}{2} \log q + (b_2 + b_4)e^{-1937}.
\end{aligned}
\end{equation}
Note that by Lemma \ref{zetabound} and the inequality $\log(1+x) \leq x$, we have that
\begin{equation}\label{lazyzeta}
    \frac{b_0}{2\eta}\log(\zeta(1+\eta)) \leq \frac{b_0}{2\eta}\log\Big(\frac{1}{\eta}\Big) + 0.3 b_0.
\end{equation}
Since $\log(1/\eta) \sim \frac{2}{3}\log\log \gamma$, it follows  that the main term on the right-hand side of \eqref{lazykey} is
\begin{equation}\label{lazymainterm}
\begin{aligned}
        &\frac{1}{\eta}\Big(\frac{b_5}{2}\Big(\frac{2}{3}\log\log \gamma + B\eta^{3/2} \log \gamma \Big) + \frac{b_0}{3}\log\log \gamma\Big) = \frac{\log\log \gamma}{\eta}\Big(b_5\Big(\frac{1}{3} + \frac{C}{2} \Big) + \frac{b_0}{3} \Big) \\
        &\qquad= C^{-2/3}\Big(b_5\Big(\frac{1}{3} + \frac{C}{2} \Big) + \frac{b_0}{3} \Big)  B^{2/3} (\log \gamma)^{2/3} (\log \log \gamma)^{1/3} \\
        &\qquad\leq 30.26576 B^{2/3} (\log \gamma)^{2/3} (\log\log \gamma)^{1/3}
\end{aligned}
\end{equation}
via \eqref{lazyzeta}.  Note that since $\frac{\log T_0}{\log\log T_0} \geq \frac{5110.6}{B}$ and $\gamma \geq T_0 \geq e^{10\,650}$,
\begin{equation}\label{3b0}
    0.3b_0 + (b_2 + b_4)e^{-1937} \leq 0.010122B^{2/3}\Big(\frac{\log \gamma}{\log \log \gamma}\Big)^{2/3}.
\end{equation}
Consequently, by \eqref{3b0}, \eqref{lazyzeta}, and the assumption $B \leq 4.45$, all of the remaining terms in \eqref{lazykey} excluding $\frac{b_5}{2}\log q$ contribute at most
{\small\begin{equation*}
    \begin{aligned}
        &\frac{b_5}{2\eta}\log(A+1) + \frac{b_0}{3\eta}(\log(B/C) -\log_3 \gamma ) + 0.010122 B^{2/3} \Big(\frac{\log\gamma}{\log\log\gamma} \Big)^{2/3} \\
        &\leq \Big(C^{-2/3}\Big(\frac{b_5}{2}\log(A+1) + \frac{b_0}{3}(\log(B/C) - \log_3 \gamma) \Big)  + 0.010122 \Big) B^{2/3} \Big(\frac{\log\gamma}{\log\log\gamma} \Big)^{2/3} \\
        &\leq (-2.7545\log_3 \gamma + 13.75563\log(A+1) + 3.33)B^{2/3} \Big(\frac{\log\gamma}{\log\log\gamma} \Big)^{2/3}.
    \end{aligned}
\end{equation*}}%
          
Consequently, we have that
\begin{align*}
    0.953\frac{1}{1-\beta} &\leq 16.66375\log q + 30.26576  B^{2/3} (\log \gamma)^{2/3} (\log \log \gamma)^{1/3} \\
    &\quad+(-2.7545\log_3 \gamma + 13.75563\log(A+1) + 3.33)B^{2/3}\Big(\frac{\log \gamma}{\log \log \gamma}\Big)^{2/3},
\end{align*}
meaning that
{\small
\[
    \frac{1}{1-\beta} \leq  \Big(31.76 + \frac{-2.89\log_3 \gamma + 14.44\log(A+1) + 3.495}{\log \log \gamma} \Big) B^{2/3}(\log \gamma)^{2/3} (\log \log \gamma)^{1/3} + 17.49 \log q ,
\]}
a contradiction.
\end{proof}
\begin{corollary}\label{inexplicitcorollary}
Let $q \geq 3$, and let $\chi\pmod{q}$ be a Dirichlet character. Then, $L(\sigma+it,\chi)$ is nonvanishing in the regions
{\small\begin{align*}
\sigma &\geq 1- \frac{1}{18\log q + 86(\log |t|)^{\frac{2}{3}}(\log\log|t|)^{\frac{1}{3}}},\qquad \abs{t} \geq \exp(\exp(\exp(23))), \\
\sigma &\geq 1- \frac{1}{18\log q + 104(\log |t|)^{\frac{2}{3}}(\log\log|t|)^{\frac{1}{3}}},\qquad \abs{t} \geq 3.
\end{align*}}%
\end{corollary}
\begin{proof}
The first part 
follows 
from Theorem \ref{inexplicit} with $B = 4.45$. For the second 
part, we take $A = 76.2$ and $B = 4.45$. 
If $T_0 \geq e^{11\,450}$, then $-2.89\log_3 t + 14.435\log(A+1) + 3.495 \leq 59.8$. Thus, by Theorem \ref{inexplicit}, we obtain $c \leq 31.76 + \frac{59.8}{\log\log T_0}$. The desired result follows once we take $T_0 = e^{11\,450}$ and apply \eqref{mccurley} for $t \leq T_0$.
\end{proof}
\bibliographystyle{abbrv}
\bibliography{mybib}
\end{document}